\documentclass[11pt]{amsart}
\usepackage[margin=1in]{geometry}

\usepackage{amsthm,amsmath,amssymb}
\usepackage{tikz,xcolor,graphicx}
\usepackage{mathtools}
\usepackage{enumitem}
\usepackage{booktabs}
\usepackage{pst-node}
\usepackage{blindtext}
\usepackage{pst-text}
\usepackage{tikz-cd}
\usepackage[unicode]{hyperref}
\usepackage{url}
\usepackage[T1]{fontenc}

\renewcommand{\AA}{\mathbb{A}}

\newcommand{\CC}{\mathbb{C}}

\newcommand{\FF}{\mathbb{F}}

\newcommand{\QQ}{\mathbb{Q}}

\newcommand{\ZZ}{\mathbb{Z}}

\newcommand{\cC}{\mathcal{C}}

\newcommand{\cO}{\mathcal{O}}

\newcommand{\fm}{\mathfrak{m}}

\newcommand{\NP}{\mathrm{NP}}
\usepackage{amsthm,amsmath,amssymb}
\usepackage{tikz,xcolor,graphicx}
\usepackage{mathtools}
\usepackage{enumitem}
\setlist[enumerate,itemize]{leftmargin=*}
\usepackage{booktabs}
\usepackage{pst-node}
\usepackage{tikz-cd}
\usepackage[unicode]{hyperref}
\usepackage{cleveref}
\usepackage{url}
\usepackage[citestyle=alphabetic-verb,bibstyle=alphabetic,firstinits=true,backend=bibtex]{biblatex}
\hypersetup{
    colorlinks=true,
    linkcolor=blue,
    filecolor=magenta,      
    urlcolor=cyan,
    pdftitle={Overleaf Example},
    pdfpagemode=FullScreen,
    }

\usetikzlibrary{calc}
\usetikzlibrary{positioning}

\newtheorem{theorem}{Theorem}[section]
\newtheorem{proposition}[theorem]{Proposition}
\newtheorem{lemma}[theorem]{Lemma}

\theoremstyle{definition}
\newtheorem{definition}[theorem]{Definition}

\newtheorem{conjecture}[theorem]{Conjecture}

\theoremstyle{remark}
\newtheorem{remark}[theorem]{Remark}
\newtheorem{notation}[theorem]{Notation}

\newtheorem{algorithm}[theorem]{Algorithm}

\numberwithin{equation}{theorem}

\DeclareMathOperator{\Hom}{Hom}

\DeclareMathOperator{\Gal}{Gal}
\DeclareMathOperator{\GL}{GL}
\DeclareMathOperator{\Sym}{Sym}

\DeclareMathOperator{\Proj}{Proj}

\DeclareMathOperator{\NS}{NS}
\DeclareMathOperator{\Ind}{Ind}

\title{Slopes of modular forms and the ghost conjecture}
\begin{document}
\author{Eunsu Hur}
\maketitle
\begin{abstract}
    We give an algorithm to compute the slope sequence of modular forms with fixed Galois components from its first few entries, which is a refined version of the conjecture of \cite{buzconj}. We use the results of Liu et al. on the ghost conjecture from \cite{ghost1}. These symmetries in slope sequences have potential implication to unexplained symmetries in many Coleman-Mazur eigencurves.
\end{abstract}
\section{Introduction}

In this paper, we prove a variant of the slope conjecture of Buzzard \cite{buzconj} which computes the slope sequence of modular forms with a fixed Galois component. We prove how one can obtain the full sequence from only having the first few entries of the sequence at hand, with an algorithm that is polynomial. Theoretically, Buzzard's conjecture predicts the slope sequence of modular forms given by the operator $T_p$ on the space $S_k(\Gamma_0(N))$ and has concrete implications of symmetry in many Colemna-Mazur eigencurves as noted in \cite{buzconj}. We progress by constructing the slope sequences inductively using patterns of the Ghost series. This is done using the results of \cite{liu2023slopes} which proves many cases of conjectures given by \cite{ghost1} and \cite{ghost2}. Our approach uses the ghost theorem in \cite{liu2023slopes} and various combinatorial properties of the ghost series. Our main theorem is as follows. Denote by $v^{(\epsilon)}_{\bar r}(k)$ the sequence of slopes obtained for weight $k$ by a variant of Buzzard's algorithm with given input of the dimension of spaces of modular forms with even weight $k<p+3$ and character $\epsilon$ with Galois component $\bar r$. Then, we have the following:

    \begin{theorem}[$\bar r$-Slope theorem]\label{mainthm}
        Let $p\geq11$ a prime, 
            \begin{equation}
                \bar r:\Gal_\QQ\to\GL_2(\FF)
            \end{equation}
        an absolutely irreducible representation such that $\bar r|_{\Gal_{\QQ_p}}$ is reducible and 
            \begin{equation}
                \bar r|_{I_{\QQ_p}}\simeq\begin{pmatrix}\omega_1^{a+b+1} & *\neq0\\0&\omega_1^b\end{pmatrix}
            \end{equation}
         with $a\in\{2,\ldots,p-5\}$ and $b\in\{0,\ldots,p-2\}$. Then for all $k\equiv a+2b+2\mod p-1$, the sequence given by $v^{(\epsilon=1\times\omega^{a+2b})}_{\bar r}(k)$ equals the sequence of slopes of the space of modular forms $S_k(\Gamma_0(N))_{\bar r}$ in increasing order.
    \end{theorem} 

In the proof the theorem, we use the following recent results of \cite{liu2023slopes}. First, we let $m(\bar r)$ be the constant 

    \begin{equation}
        m(\bar r)=\frac{(p-1)\dim S_k(\Gamma_0(Np);\omega^{k-1-c})_{\bar r}}{2k}.
    \end{equation}
We note that \cref{mainthm} is in the form of classical modular forms, but all but the appendix of the paper is written in the language of abstract modular forms, which include $\GL_2/\QQ$ automorphic forms and we will shift our notation starting in \cref{properties}.

    \begin{theorem}[Ghost theorem, ignoring the case when $\bar\rho$ is split]
        Assume $p\geq 11.$ If $\bar r$ satisfies the conditions imposed in the theorem above. Then for every $w_\star\in\fm_{\CC_p}$, the Newton polygon\footnote{defined in Definition~\ref{NP}} $\NP(C_{\bar r}(w_\star,-))$, where $C_{\bar r}$ is the characteristic power series of $\bar r$-locallized weight $k$ overconvergent modular forms when $w_\star=w_k=\exp(p(k-2))$, is the same as the Newton polygon $\NP(G_{\bar\rho}(w_\star,-))$, stretched by $m(\bar r)$ (except possibly for their slope zero parts which is the case for when $\bar\rho$ is split). 
    \end{theorem}
We write the definition of the ghost series here for completeness, but we will define them in the more abstract setting later. The formulation in the introduction is only for the interest of classical modular forms, and simplicity.
    \begin{definition}[Definition of the ghost series]\label{ghostseries}
        Assume that $\bar r|_{I_{\QQ_p}}\simeq\bar\rho$. For each $k\equiv a+2b+2 \mod p-1$ and $k\geq 2$, define 
            $$d_k^{ur}:=\frac{1}{m(\bar r)}\dim S_k(\Gamma_0(N))_{\bar r},\ d_k^{Iw}=\frac{1}{m(\bar r)}\dim S_k(\Gamma_0(Np))_{\bar r}.$$ 
        Then we have
            $$g_n(w)=\prod_{k\equiv a+2b+2\mod p-1}(w-w_k)^{m_n(k)}$$
        where $$m_n(k)=\begin{cases}
                \min(n-d_k^{ur},d_k^{Iw}-d_k^{ur}-n)&\text{if}d_k^{ur}<n<d_k^{Iw}-d_k^{ur}\\0&\text{otherwise}.
            \end{cases}$$
             Then we define the ghost series as $G_{\bar\rho}(w,t)=1+\sum_{n\geq1}g_n(w)t^n\in\ZZ_p[w][[t]]$.
    \end{definition}
The ghost series depends on the dimension of the $\bar r$-components of the space of modular forms. We will recall formulas of these in \cref{properties} along with other parts of the ghost conjecture. In \cref{section3}, we prove various properties of the Newton polygon of the Ghost series, following \cite{liu2023slopes}. In \cref{varalg}, we state the variant of the Slope conjecture of \cite{buzconj} which is the main theorem of this paper, which we prove in \cref{mainproof}.
\subsection{Acknowledgment} This work was done while the author was at Imperial College London, under the supervision of Professor Toby Gee, funded by the MIT-MISTI UK program. The author would like to thank professor Toby Gee for suggesting the project, for the guidance and support throughout the project.
\section{Recollections from the Ghost conjecture}\label{properties}
\subsection{Notations}
We will recall the most recent form of the ghost conjecture proven in \cite{liu2023slopes}. First we start with some notations. Let $p\geq5 $ be an odd prime and fix an isomorphism $\bar\QQ_p\simeq\CC$. Let $E/\QQ_p$ be a finite extension and $\cO$ and $\FF$ be its ring of integers and residue field. Let $\bar r:\Gal_\QQ\to\GL_2(\FF)$ be an absolutely irreducible representation that satisfies 
\begin{equation}\label{galois}
    \bar r|_{I_{\QQ_p}}\simeq\begin{pmatrix}
        \omega_1^{a+b+1} & *\neq 0\\
        0 & \omega_1^b
    \end{pmatrix}
\end{equation}with $a\in\{1,\cdots,p-4\}$ and $b\in\{0,\cdots,p-2\}$. 
We first define the notion of a Newton polygon.

    \begin{definition}\label{NP}
        Let $f(t)=\sum_{n\geq0}a_nt^n\in\cO[[t]]$. Then, we define the Newton polygon $\NP(f)$ of $f$ as the convex polygon (possibly infinite) given by taking the lower convex hull of the points $(n,v_p(a_n))$ for $n\in\ZZ_{\geq0}$ where $v_p$ denotes the $p$-adic valuation. We also define to stretch a newton polygon of $f$ by a factor of $m$ by taking $\NP(\sum_{n\geq 0}a_n^mt^{nm})$.
    \end{definition}
    
We recall notations from \cite{liu2023slopes} where they prove various results on abstract modular forms which are crucial to our result. We work in a more abstract setup and will specialize to our case later. We first recall the following subgroups of $\GL_2(\QQ_p)$:
    $$K_p:=\GL_2(\ZZ_p)\supset Iw_p:=\begin{pmatrix}\ZZ_p^\times&\ZZ_p\\p\ZZ_p&\ZZ_p^\times\end{pmatrix}\supset Iw_{p,1}:=\begin{pmatrix}1+p\ZZ_p&\ZZ_p\\p\ZZ_p&1+p\ZZ_p\end{pmatrix}.$$
Fix $\bar\rho$ a reducible, nonsplit, and generic residual representation:

    \begin{equation}
        \bar\rho\simeq\begin{pmatrix}\omega_1^{a+b+1}&*\neq0\\0&\omega_1^b\end{pmatrix}\text{ for }1\leq a\leq p-4 \text{ and }0\leq b\leq p-2.
    \end{equation}

We point to Section~2 of \cite{liu2023slopes} for specific notations and definitions of abstract modular forms. Let $\omega\colon\FF^\times_p\to\cO^\times$ be the Teichm\"uller lift. A character $\epsilon\colon(\FF_p^\times)^2\to\cO^\times$ is called \emph{relevant to $\bar\rho$} if it is the form $\epsilon=\omega^{-s_\epsilon+b}\times\omega^{a-s_\epsilon+b}$ for some $s_\epsilon\in\{0,\ldots,p-2\}$. We will follow the notion of  projective augmented $\cO[[K_p]]$-module as in Definition 2.2. of \cite{liu2023slopes}. 

    \begin{definition}\cite[Definition 2.2.]{liu2023slopes}
        Define $\tilde H$ to be a projective augmented $\cO[[K_p]]$-module of $\bar\rho$ type and multiplicity $m(\tilde H)$ if $\tilde H$ is a finitely generated right projective $\cO[[K_p]]$-module whose right $K_p$ action extends to a right continuous $\GL_2(\QQ_p)$ action and moreover $\overline H=\tilde H/(\varpi,I_{1+pM_2(\ZZ_p)})$ is isomorphic to a direct sum of $m(\tilde H)$ copies of $\Proj_{a,b}$ as a right $\FF[\GL_2(\FF_p)]$-module (See Appendix~A of \cite{liu2022local} for a detailed definition of $\Proj_{a,b}$). We say that $\tilde H$ is \emph{primitive} if $m(\tilde H)=1$.
            
    \end{definition}
Let $\tilde H$ be a projective augmented module, then as in \cite{liu2023slopes} we can define the space of modular forms as follows.

    \begin{definition}
        We define the space of $p$-adic modular forms and overconvergent modular forms by
        \begin{equation}
    \begin{split}
        & S_{p-adic}^{(\epsilon)}=S_{\tilde{H},p-adic}^{(\epsilon)}:=\Hom_{\cO[Iw_p]}(\tilde H,\cC^0(\ZZ_p;\cO[[w]]^{(\epsilon)}))\\
        \text{and}\\
        & S^{\dagger,(\epsilon)}=S_{\tilde{H}}^{\dagger,(\epsilon)}:=\Hom_{\cO[Iw_p]}(\tilde H,\cC^0(\ZZ_p;\cO\langle w/p\rangle^{(\epsilon)}\langle z\rangle)).
    \end{split}
\end{equation}
    Where the action of $Iw_p$ on $\cC^0(\ZZ_p,\cO[[w]]^{(\epsilon)})$ is given by defining
        \begin{equation}
            \begin{split}
                & \chi_{univ}^{(\epsilon)}:\FF_p^\times\times\ZZ_p^\times\to\cO[[w]]^{(\epsilon),\times}\\
                &(\bar\alpha,\delta)\mapsto\epsilon(\bar\alpha,\bar\delta)\cdot(1+w)^{\log(\delta/\omega(\bar\delta))/p}
            \end{split}
        \end{equation}
    and identifying 
        \begin{equation}
            \Ind_{B^{op}(\ZZ_p)}^{Iw_p}\chi_{univ}^{(\epsilon)}\simeq\cC^0(\ZZ_p;\cO[[w]]^{(\epsilon)}).        \end{equation}
    This also gives an action on $\cO\langle w/p\rangle^{(\epsilon)}\langle z\rangle$ viewing power series as a continuous function.
    \end{definition}

Here we can extend the action of $Iw_p$ on $\cC^0(\ZZ_p;\cO[[w]]^{(\epsilon)})$ and $\cO\langle w/p\rangle^{(\epsilon)}\langle z\rangle$ to 
\begin{equation*}
    M_1=\bigg\{\begin{pmatrix}\alpha&\beta\\ \gamma&\delta\end{pmatrix}\in M_2(\ZZ_p);p|\gamma,p\nmid\delta,\alpha\delta-\beta\gamma\neq0\bigg\}
\end{equation*} 
by\footnote{We indicate $g\in M_1$
acting on $h(z)$ by $h|_g(z)$.}
\begin{equation*}
    h|_{\begin{pmatrix}\alpha&\beta\\ \gamma&\delta\end{pmatrix}}(z)=\epsilon(\bar d/\bar\delta)\cdot(1+w)^{\log((\gamma z+\delta)/\omega(\bar\delta))/p}\cdot h\bigg(\frac{\alpha z+\beta}{\gamma z+\delta}\bigg).
\end{equation*}
Using this, we can define the $U_p$ operator as follows. Recall the decomposition
\[Iw_p\begin{pmatrix}p^{-1}&0\\0&1\end{pmatrix}Iw_p=\coprod_{j=0}^{p-1}v_jIw_p\]where $v_j=\begin{pmatrix}p^{-1}&0\\j&1\end{pmatrix}$. The $U_p$ operator sends $\varphi\in S^{\dagger,(\epsilon)}$ to 
    $$U_p(\varphi)(x)=\sum_{j=0}^{p-1}\varphi(xv_j)|_{v_j^{-1}}.$$
Thus, we can define the characteristic power series:

    \begin{equation}
        C^{(\epsilon)}(w,t)=C_{\tilde H}^{(\epsilon)}(w,t):=\det(1-U_pt|S^{\dagger,(\epsilon)})=\sum_{n\geq0}c_n^{(\epsilon)}(w)t^n\in\Lambda[[t]]=\cO[[w,t]].
    \end{equation}
    
Now, if we let $\psi=\epsilon\cdot(1\times\omega{^{2-k}})$, we have $S_k^\dagger(\psi)=S^{\dagger,(\epsilon)}\otimes_{\cO\langle w/p\rangle,w\mapsto w_k}\cO$ carrying compatible $U_p$ actions. Moreover, the characteristic power series of the $U_p$ action is $C^{(\epsilon)}(w_k,t)$. \\
For each $k\geq 2$, setting $\psi$ as above, we have an inclusion $$\cO[z]^{\deg\leq k-2}\otimes\psi\subset\cO\langle w/p\rangle^{(\epsilon)}\langle z\rangle\otimes_{\cO\langle w/p\rangle,w\mapsto w_k}\cO,$$ and via this, we define the space of abstract classical forms of weight $k$ and character $\psi$ to be
    $$S_k^{Iw}(\psi)=\Hom_{\cO[Iw_p]}(\tilde H,\cO[z]^{\leq k-2}\otimes\psi)\subset S_k^\dagger(\psi).$$
When $\tilde H$ is primitive, i.e. $m(\tilde H)=1$, we define $$d_k^{Iw}(\psi)=\mathrm{rank}_\cO S_k^{Iw}(\psi).$$
Now, we fix a relevant character $\epsilon,$ and define $k_\epsilon=2+\{a+2s_\epsilon\}\in\{2,\ldots,p\}$ where $\{-\}$ denotes the remainder modulo $ p-1$. Also for a character $\epsilon\colon(\FF_p^\times)^2\to\cO^\times$, define $\epsilon_1$ to be the projection on to the first factor $\FF_p^\times$, and for any character $\chi\colon\FF_p^\times\to\cO^\times$, let $\tilde\chi=\chi\times\chi\colon(\FF_p^\times)^2\to\cO^\times$ a character of $(\FF_p^\times)^2$. When $\psi$ is of the form $\psi=\tilde\epsilon_1=\epsilon_1\times\epsilon_1$, and $k$ satisfies $\tilde\epsilon_1=\epsilon\cdot(1\times\omega^{2-k})=\omega^{-s_\epsilon+b}\times\omega^{a+s_\epsilon+b+2-k}$, we must have $k\equiv k_\epsilon\mod p-1$. In such case, $\cO[z]^{\leq k-2}\otimes\epsilon_1\circ\det$ has a natural action of $M_1$, and hence we can define
    $S_k^{ur}(\epsilon_1)=\Hom_{\cO[K_p]}(\tilde H,\cO[z]^{\leq k-2}\otimes\epsilon_1\circ\det).$
For each relevant character $\epsilon=\omega^{-s_\epsilon+b}\times\omega^{a+s_\epsilon+b}$, we set $\tilde\epsilon_1=\omega^{-s_\epsilon+b}\times\omega^{-s_\epsilon+b}$. Assuming $\tilde H$ is primitve, $d_k^{ur}(\epsilon_1):=\mathrm{rank}_\cO S_k^{ur}(\epsilon_1)$. When $\tilde H$ is not primitive, we define the dimension functions by $d_k^{Iw}(\psi)=\frac{1}{m(\tilde H)}\mathrm{rank} S_{\tilde H,k}^{Iw}(\psi)$ and $d_k^{ur}(\epsilon_1)=\frac{1}{m(\tilde H)}\mathrm{rank} S_{\tilde H,k}^{ur}(\epsilon_1)$. 
Following the notations of \cite{liu2023slopes}, we define the ghost series of type $\bar\rho$: 

\begin{definition}
    The \emph{ghost series} associated to $\bar r$ with character $\epsilon$ is 
        \begin{equation}
            G^{(\epsilon)}(w,t)=G_{\bar\rho}^{(\epsilon)}(w,t)=1+\sum_{n=1}^\infty g_n^{(\epsilon)}(w)t^n\in\cO[[w,t]],
        \end{equation}
        where 
\begin{equation}
    g_n^{(\epsilon)}(w)=\prod_{k\geq 2,k\equiv k_\epsilon\mod p-1}(w-w_k)^{m_n^{(\epsilon)}(k)}\in\cO[w]
\end{equation} 
with $m_n^{(\epsilon)}(k)$ given by
    \begin{equation}
        m_n^{(\epsilon)}(k)=\begin{cases}
        \min\{n-d_k^{ur}(\epsilon_1),d_k^{Iw}(\tilde\epsilon_1)-d_k^{ur}(\epsilon_1)-n\}&\text{if }d_k^{ur}(\epsilon_1)<n<d_k^{Iw}(\tilde\epsilon_1)-d_k^{ur}(\epsilon_1)\\
        0&\text{otherwise}.
    \end{cases}
    \end{equation}
\end{definition}

\subsection{Recollections of properties of the ghost series}
Given the setup above, we have the following propositions (Proposition 2.16 of \cite{liu2023slopes}). We would like to point out that from now on, we will assume that $p\geq 11$ since later we will identify the Newton polygons of the ghost series and the characteristic power series using Theorem~\ref{ghostthm} which holds for $p\geq 11$.\\

To make it easier to visualize, we set up some notation and will write the above proposition in the new notation. Let $v^{(\epsilon),\dagger}_k[n]$ be the $n$th slope of the Newton polygon $\NP(G^{(\epsilon)}_{\bar \rho}(w_k,-))$ where $\bar r$ is of type $\bar\rho$. We will write denote by $v^{(\epsilon),Iw}_k[n],v^{(\epsilon)}_k[n]$ the slopes sequences for different spaces of modular forms respectively where this notation is to resemble the notation of \cite{buzconj}. We will drop the $\bar r$ in the notation as we will work with a fixed $\bar r$ until the appendix. Then we have the following.

    \begin{proposition}\label{ghostprops}
        Let $\epsilon$ be a relevant character. Fix $k_0\geq 2$, write $$g_{n,\hat k}^{(\epsilon)}(w):=g_n^{(\epsilon)}(w)/(w-w_k)^{m_n^{(\epsilon)}(k)}.$$ We let $d:=d_{k_0}^{Iw}(\epsilon\cdot(1\times\omega^{2-k_0}))$ in this proposition.
            \begin{enumerate}
                \item (Compatibility with theta maps) Put $\epsilon'=\epsilon\cdot(\omega^{k_0-1}\times\omega^{1-k_0})$ with $s_{\epsilon'}=\{s_\epsilon+1-k_0\}$. Then for every $l\geq 1$, 
                    \begin{equation}\label{theta}
                        v^{(\epsilon),\dagger}_{k_0}[d+l]=v^{(\epsilon'),\dagger}_{2-k_0}[l]+k_0-1.
                    \end{equation}
                \item (Compatibility with Atkin-Lehner involutions) Assume that $k\not\equiv k_\epsilon\mod p-1$. Put $\epsilon''=\omega^{-s_{\epsilon''}}\times\omega^{a+s_{\epsilon''}}$ with $s_{\epsilon''}:=\{k_0-2-a-s_\epsilon\}$. Then for every $l\in\{1,\ldots,d\}$, 
                    \begin{equation}\label{atkinlehner}
                        v^{(\epsilon),\dagger}_{k_0}[l]+v^{(\epsilon''),\dagger}_{k_0}[d-l+1]=k_0-1.
                    \end{equation}
                \item (Compatibility with $p$-stabilizations) Assume that $k_0\equiv k_\epsilon\mod p-1$. Then for every $l\in\{1,\ldots,d_{k_0}^{ur}(\epsilon_1)\}$, 
                    \begin{equation}\label{pstab}
                        v^{(\epsilon),\dagger}_{k_0}[l]+v^{(\epsilon),\dagger}_{k_0}[d-l+1]=k_0-1.
                    \end{equation}
                \item (Ghost duality) Assume $k_0\equiv k_\epsilon\mod p-1$. Then for each $l=0,\ldots,\frac{1}{2}d_{k_0}^{new}(\epsilon_1)-1$, 
                    \begin{equation}\label{ghostdual}
                        v_p(g_{d_{k_0}^{Iw}(\tilde\epsilon_1)-d_{k_0}^{ur}(\epsilon_1)-l,\hat k_0}(w_{k_0}))-v_p(g_{d_{k_0}^{ur}+l(\epsilon_1)-l,\hat k_0}(w_{k_0}))=(k_0-2)\cdot(\frac{1}{2}d_{k_0}^{new}(\tilde\epsilon_1)-l).
                    \end{equation}
            \end{enumerate}
    \end{proposition}

We record dimension formulas for later use.

    \begin{proposition}[Proposition 2.12 \cite{liu2023slopes}]\label{dimform}
        Let $\tilde H$ be a primitive $\cO[[K_p]]$-projective augmented module of type $\bar\rho$ and let $\epsilon=\omega^{-s_\epsilon+b}\times\omega^{a+b+s_\epsilon}$ be a relevant character of $(\FF_p^\times)^2$.
            \begin{enumerate}
                \item We have
                    \begin{equation}
                        d_k^{Iw}(\epsilon\cdot(1\times\omega^{2-k}))=\bigg\lfloor \frac{k-2-s_{\epsilon}}{p-1}\bigg\rfloor+\bigg\lfloor\frac{k-2-\{a+s_\epsilon\}}{p-1}\bigg\rfloor+2
                    \end{equation}
                \item Set $\delta_\epsilon:=\bigl\lfloor\frac{s_\epsilon+\{a+s_\epsilon\}}{p-1}\bigr\rfloor$. In particular when $k=k_\epsilon+(p-1)k_\bullet$, for $k_\bullet\in\ZZ_{\geq0}$, we have 
                    $$d_k^{Iw}=2k_\bullet+2-2\delta_\epsilon$$
                \item Define two integers $t_1,t_2\in\ZZ$ as follows.
                    \begin{itemize}
                        \item If $a+s_\epsilon<p-1$, let $t_1=s_\epsilon+\delta_\epsilon$ and $t_2=a+s_\epsilon+\delta_\epsilon+2$
                        \item If $a+s_\epsilon\geq p-1$, let $t_1=\{a+s_\epsilon\}+\delta_\epsilon+1$ and  $t_2=s_\epsilon+\delta_\epsilon+1$.
                    \end{itemize}
                    Then for $k=k_0+(p-1)k_\bullet$, 
                        \begin{equation}
                            d_k^{ur}=\bigg\lfloor\frac{k_\bullet-t_1}{p+1}\bigg\rfloor+\bigg\lfloor\frac{k_\bullet-t_2}{p+1}\bigg\rfloor+2.
                        \end{equation}
            \end{enumerate}
    \end{proposition}
Finally, the main theorem of \cite{liu2023slopes} is the following.
    \begin{theorem}[Theorem 8.7 of \cite{liu2023slopes} for $\bar\rho$ non-split]\label{ghostthm}
        Assume that we have $p,\tilde H,\epsilon,\bar r$,$C_{\tilde H}^{(\epsilon)}(w,t)$, $G_{\bar\rho}^{(\epsilon)}(w,t)$ as above. 
        Then for every $w_\star\in\fm_{\CC_p}$, the Newton polygon $\NP(C_{\tilde H}^{(\epsilon)}(w_\star,-))$ is the same as the Newton polygon $\NP(G_{\bar \rho}^{(\epsilon)}(w_\star,-))$, stretched by $m(\tilde H)$ in the language of \cref{NP}.
    \end{theorem}
This theorem will later fit into the proof of the main theorem in \cref{GuMaz} by directly deducing \cite[Theorem 8.10]{liu2023slopes}.

\section{Newton polygon of the Ghost series}\label{section3}

We recall some lemmas related to the vertices of the Newton polygon of the Ghost series. All of the material is from \cite{liu2022local}.

    \begin{notation}
        For any integer $k\geq 2$ and $k\equiv k_\epsilon\mod p-1$, we set 
            $$\Delta_{k,l}':=v_p(g^{(\epsilon)}_{\frac{1}{2}d_l^{Iw}+l,\hat k}(w_k))-\frac{k-2}{2}l,\text{ for }l=-\frac{1}{2}d_k^{new},\ldots,\frac{1}{2}d_k^{new}$$
        
    \end{notation}
Then the ghost duality theorem \cref{ghostdual} says $\Delta_{k,l}'=\Delta_{k,-l}'$.

    \begin{definition}
        We define $\Delta_k$ to be the convex hull of  $(l,\Delta_{k,l}')$ and denote the corresponding points $(l,\Delta_{k,l})$ to be the points lying on $\Delta_k$. 
    \end{definition}

    \begin{lemma}[Lemma 5.2 in \cite{liu2022local}]\label{Delta'bound}
        For $k=k_\epsilon+(p-1)k_\bullet$ and $l=1,\ldots,\frac{1}{2}d_k^{new}$, we have
            \begin{equation}
                \Delta_{k,l}'-\Delta_{k,l-1}'\geq\frac{3}{2}+\frac{p-1}{2}(l-1).
            \end{equation}
    \end{lemma}
    \begin{lemma}[Lemma 5.8 in \cite{liu2022local}]\label{diffDD'}
        Assume $p\geq 7$. For $k=k_\epsilon+(p-1)k_\bullet$ and $l=1,\ldots,\frac{1}{2}d_k^{new}$, we have 
            \begin{equation}
                \Delta_{k,l}'-\Delta_{k,l}\leq 3(\log l/\log p)^2.
            \end{equation}
        Moreover, we have the following: when $l<2p$, $\Delta_{k,l}'=\Delta_{k,l}$ if $l\neq p$, if $l=p$ then $\Delta_{k,l}'-\Delta_{k,l}\leq1$.
    \end{lemma}
Using the two lemmas above, we get 
    \begin{lemma}
        Let $p\geq 7$, $l,k$ as above.
            \begin{equation}\label{boundondifDelta}
                \Delta_{k,l}-\Delta_{k,l-1}\geq l.
            \end{equation}
    \end{lemma}

    \begin{proof}
        The proof is to combine the two lemmas above to get the desired inequality. We divide into three cases.
            \begin{enumerate}
                \item $l<2p,l\neq p$: Then the result is clear by \cref{Delta'bound}.
                \item $l=p$: Then by \cref{diffDD'} and \cref{Delta'bound}, we get 
                    \begin{equation}
                        \Delta_{k,l}-\Delta_{k,l-1}\geq\frac{3}{2}+\frac{p-1}{2}(l-1)-2\geq l
                    \end{equation} 
                for our conditions.
                \item $l\geq 2p$: We have that $\Delta_{k,l}'-\Delta_{k,l}\leq 3(\log l/\log p)^2$, hence we have a bound 
                    \begin{equation}
                        \Delta_{k,l}-\Delta_{k,l-1}\geq\frac{3}{2}+\frac{p-1}{2}(l-1)-6(\log l/\log p)^2
                    \end{equation}
                and taking thinking of the right hand side as a function taking real values as inputs, taking derivatives, we get 
                    \begin{equation}
                        \frac{d}{dl}\bigg(\frac{3}{2}+\frac{p-1}{2}(l-1)-6(\log l/\log p)^2\bigg)\geq \frac{p-1}{2}-12\log l/l(\log p)^2>1
                    \end{equation}
                which all holds from $p\geq11,l\geq 2p$. Hence we get the desired result.\qedhere
            \end{enumerate}
    \end{proof}
    \begin{definition}\label{nearSt}
        Let $w_\star\in\fm_{\CC_p}.$ For each $k=k_\epsilon+(p-1)k_\bullet$, let $L_{w_\star,k}$ denote the largest number in $\{1,\ldots,\frac{1}{2}d_k^{new}\}$ such that
            \begin{equation}
                 v_p(w_\star-w_k)\geq\Delta_{k,L_{w_\star,k}}-\Delta_{k,L_{w_\star,k}-1}
            \end{equation}
        and call the open interval 
            \begin{equation}
                \NS_{w_\star,k}=\big(\frac{1}{2}d_k^{Iw}-L_{w_\star,k},\frac{1}{2}d_k^{Iw}+L_{w_\star,k}\big)
            \end{equation} 
        the near-Steinberg range for the pair $(w_\star,k)$ following the definition of \cite{liu2022local}. When no such $L_{w_\star,k}$ exists, we define $\NS_{w_\star,k}=\emptyset$. For a positive integer $n$, we say $(w_\star,n)$ is near-Steinberg if $n$ belongs to the near-Steinberg range $\NS_{w_\star,k}$ for some $k$. 
    \end{definition}

We state the main theorem of \cite[Theorem 5.19]{liu2022local}.

    \begin{theorem}\label{NPpts}
        Fix a relevant character $\epsilon$ and $w_\star\in\fm_{\CC_p}$.
            \begin{enumerate}
                \item The set of near-Steinberg ranges $\NS_{w_\star,k}$ for all $k$ is nested, i.e. for any two such open intervals, either they are disjoint or one is contained in another.\\ A near-Steinberg range $\NS_{w_\star,k}$ is called maximal if it is not contained in other near-Steinberg ranges.
                \item The $x$-coordinates of the vertices of the Newton polygon $\NP(G^{(\epsilon)}(w_\star,-))$ are exactly those integers which do not lie in any $\NS_{w_\star,k}$. Equivalently, for an integer $n\geq1$, the pair $(n,w_\star)$ is near-Steinberg if and only the point $(n,v_p(g^{(\epsilon)}_n(w_\star)))$ is not a vertex of $\NP(G^{(\epsilon)}(w_\star,-))$.
            \end{enumerate}
    \end{theorem}

We also cite  proposition 4.1 of  \cite{liu2023slopes} for later use.

    \begin{theorem}\label{someNPpts}
     
        For a relevant character $\epsilon$, and $k\in\ZZ_{\geq2}$, writing $d_{\epsilon,k}=d_k^{Iw}(\epsilon\cdot(1\times \omega^{2-k}))$, then $(d_{\epsilon,k},v_p(c_{d_{\epsilon,k}}^{(\epsilon)}(w_k)))$ is a vertex of $\NP(C^{(\epsilon)}(w_k),-))$ and $d_{\epsilon,k},v_p(g_{d_{\epsilon,k}}^{(\epsilon)}(w_k)))$ of $\NP(G^{(\epsilon)}(w_k),-))$.

    \end{theorem}
    
\section{$\bar r$ component of the Slope conjecture}\label{varalg}

In this section we discuss how we should change Buzzard's conejcture to the setting of abstract modular forms in order to apply the ghost conjecture to prove the cases. For the original formulation of the algorithm and the version for classical modular forms, see the appendix. Recall the notation from \cref{properties}. Fix a relevant character $\epsilon=\omega^{-s_\epsilon+b}\times\omega^{a+b+s_\epsilon}$.\\
We fix $\bar r:G_\QQ\to\GL_2(\FF)$ to be an absolutely irreducible representation but reducible when restricted to the decomposition group such that
    $$\bar r|_{I_{\QQ_p}}\simeq\begin{pmatrix}
        \omega_1^{a+b+1} & *\neq0\\
        0 & \omega_1^b
    \end{pmatrix}$$ 
when restricted to the inertia group.
We set up the notations in order to define the algorithm that predicts the $T_p$ slopes. 
  

    \begin{notation}
        \begin{itemize}\leavevmode
            \item Let $\tilde H$ be a projective augmented module of type $\bar r$.
            \item Denote by $v^{(\epsilon)}_k$ the sequence of $T_p$ slopes on the space 
                \begin{equation}
                    S_k^{ur}(\epsilon_1)=S_{\tilde{H},k}^{ur}(\epsilon_1).
                \end{equation}
            \item We write a finite sequence as $s=[a_1,\ldots,a_n]$, denote $l(s)$ as its length, $s[i]$ as $a_i$. 
            \item For two sequences $a,b$, we write $a\cup b$ as the length $l(a)+l(b)$ sequence as $a$ followed by $b$. 
            \item If $l(a)=l(b)$, then $\min(a,b)$ is given by pointwise minimum.  
            \item For $n,r\geq 0,$ let $\kappa(n,r)$ to be the constant sequence of length $n$, value $r$. 
            \item If $v$ is a sequence, let $v+e$ be pointwise adding $e$ and $e-v$ be pointwise subtracting from $e$ with order reversed, i.e., \begin{equation}
                (e-v)[i]=e-v[l(v)-i+1].
            \end{equation}
            \item If $v$ has length at least $\delta$, then $\sigma(v,\delta)$ is the truncation up to $\delta$, and if $1\leq\delta_1,\delta_2\leq l(v)$, $\sigma(v,\delta_1,\delta_2)$ is the sequence cut from $\delta_1$ to $\delta_2$ (endpoints included).
            \item Let $d_{k}^{ur}(\epsilon_1)=\dim S_{\tilde H,k}^{ur}(\epsilon_1)$.
        \item Let $d_{k}^{Iw}(\psi)=\dim S_{\tilde H,k}^{Iw}(\psi)$.
        \end{itemize}
    \end{notation}

    \begin{algorithm}\label{varslopeconj}
        We start defining sequences $t^{(\epsilon)}_k$ of length $d_k^{ur}(\epsilon_1)$ (except for $k=2$) and note that $d_k^{ur}(\epsilon_1)$ is nonzero and only if $k_\epsilon\equiv k\mod p-1$ hence those are the only cases when $t^{(\epsilon)}_k$ is nonempty.\\
        We define $s^{(\epsilon)}_2=\kappa(d_2^{ur}(\epsilon_1),0)$ and $s^{(\epsilon)}_k=t^{(\epsilon)}_k$ for $k>2.$ For $4\leq k\leq p+1$, let $t^{(\epsilon)}_k=\kappa(d_k^{ur}(\epsilon_1),0)$ and $t^{(\epsilon)}_2=\kappa(d_2^{Iw}(\tilde\epsilon_1)-d_2^{ur}(\epsilon_1),0)$(note again, we are setting these sequences only for the right pairs of $\epsilon,k$). Set $k_{min}=p+3$. \\
        Now, assume that $k\geq k_{min}$ is even and we have $t_l$ for all even $l<k$. We now define $t_k$ depending on three parameters $x,y,z$. \\
        $x$ is defined as the unique positive integer such that \begin{equation*}
            p^x<k-1\leq p^{x+1}
        \end{equation*} 
        $y$ be the positive integer satisfying 
        \begin{equation*}
            p^xy<k-1\leq p^x(y+1).
        \end{equation*} 
        Set 
        \begin{equation*}
            z=1+\bigg\lfloor\frac{k-2-p^xy}{p^{x-1}}\bigg\rfloor.
        \end{equation*} 
        Then $1\leq z\leq p$. We define a sequence $V$ which are the first few slopes of $t^{(\epsilon)}_k$. The algorithm used for $V$ will depend on $y,z$ on the following three cases: $b+c\leq p-1$, $y<p-1<y+z$, and $y=p-1$.
        \begin{enumerate}
            \item When $y+z\leq p-1$: We let 
                \begin{equation*}
                    \begin{split}
                        & k_1=k-y(p-1)p^{x-1}\\
                        & k_2=k-(y-1)(p-1)p^{x-1}-2(y+z-1)p^{x-1}.
                    \end{split}
                \end{equation*} 
                Set 
                \begin{equation*}
                    v_1=t^{(\epsilon)}_{k_1},\ v_2=t^{(\epsilon'')}_{k_2}\text{ where }s_{\epsilon''}=e-1-a-s_\epsilon.
                \end{equation*}
                Define 
                \begin{equation*}
                    B=p^xy+p^{x-1}(z-1)+1,\ e=k-B.
                \end{equation*} 
                Finally set 
                \begin{equation}
                    s=1+d_{1+e}^{Iw}(\epsilon\cdot(1\times\omega^{1-e}))=1+d_{1+e}^{Iw}(\omega^{-s_\epsilon+b}\times\omega^{k-1-e+b-s_\epsilon})=1+d_{1+e}^{Iw}(\tilde\omega^{-s_\epsilon+b}\cdot(1\times\omega^{B-1})).
                \end{equation}
                If $l(v_1)\geq s-1$, then let $V=\sigma(v_1,s-1)$. \\
                Otherwise let $V_1=v_1\cup(e-\sigma(v_2,s-1-l(v_1)))$.
            \item When $y<p-1<y+z$: We set 
                \begin{equation*}
                    \begin{split}
                        & k_1=k-((y+1)p^{x-1}(p-1))\\
                        & k_2=k-p^{x-1}(p-1).
                    \end{split}
                \end{equation*}
                We let $v_1=t^{(\epsilon)}_{k_1}$ and $v_2=t^{(\epsilon)}_{k_2}$, and
                define 
                \begin{equation*}
                    B=(y+1)p^{x-1}(p-1)+1,\ e=k-B.
                \end{equation*}
                Finally set 
                \begin{equation*}
                    s=1+d_{1+e}^{Iw}(\tilde\epsilon_1),\ s_2=\lfloor (s-1)/2\rfloor,\ e_2=\lfloor e/2\rfloor.
                \end{equation*}
                \begin{enumerate}
                    \item If $l(v_1)\geq s-1$, let $V=\sigma(v_1,s-1)$.
                    \item Else if $s-1\leq 2l(v_1)<2(s-1)$, let $V=v_1\cup(e-\sigma(v_1,s-1-l(v_1)))$.
                    \item Else then define $w=\min(\sigma(v_2,l(v_1)+1,s_2),e_2)$. Let
                    \[
                    V=\begin{cases}
                        v_1\cup w\cup [e_2]\cup(e-1-w)\cup(e-v_1)&\text{if $s$ is even}\\
                        v_1\cup w\cup (e-1-w)\cup(e-v_1)&\text{if $s$ is odd}.
                    \end{cases}
                    \]
                \end{enumerate}

            \item When $y=p-1$: We let  $k_1=k-p^x(p-1)$ and $k_2=k-p^{x-1}(p-1)$, and set 
                $v_1=t^{(\epsilon)}_{k_1}$ and $v_2=t^{(\epsilon)}_{k_2}$.
                Set 
                \begin{equation*}
                    B=p^x(p-1),\ e=k-B.
                \end{equation*}
                Next, set 
                \begin{equation*}
                    s=1+d_{1+e}^{Iw}(\tilde\epsilon_1),s_2,e_2 \text{ as above.}
                \end{equation*}
                
                \begin{enumerate}
                    \item If $l(v_1)\geq s-1$, then we set $V=\sigma(v_1,s-1-l(v_1)$.
                    \item Else if $s-1\leq 2l(v_1)<2(s-1)$, let $V=v_1\cup(e-\sigma(v_1,s-1-l(v_1)))$.
                    \item Else define $w_0=\sigma(v_2,l(v_1)+1,s_2)$ and $w=\min(w_0+1,\kappa(l(w_0),e_2))$
                    \[
                    V=\begin{cases}
                        v_1\cup w\cup [e_2]\cup (e-1-w)\cup (e-v_1)&\text{if $s$ is even}\\
                        v_1\cup w\cup (e-1-w)\cup(e-v_1)&\text{if $s$ is odd}.
                    \end{cases}
                    \]
                \end{enumerate}          
        \end{enumerate}
        Now, finally we define $k_3=2B-k$ and $v_3=t^{(\epsilon')}_{k_3}$ where $\epsilon'=\epsilon\cdot(\omega^{e}\times\omega^{-e})$ and $t^{(\epsilon)}_k=\sigma(V\cup(e+v_3),d_k^{ur}(\tilde\epsilon_1))$.
    \end{algorithm}

\section{Proof of main theorem}\label{mainproof}
We first state the main theorem again in a form that is easy to see as an inductive argument. We first setup some notations for convenience.

    \begin{notation}
        Let $\tilde H$, $\epsilon$, and $\bar r$ be as before. Define $v^{(\epsilon),\dagger}_k$ and $ v^{(\epsilon),Iw}_k$ to be the sequence of slopes of the $U_p$ operator on the space $S_{\tilde H,k}^\dagger(\epsilon\cdot(1\times\omega^{2-k})$ and $S_{\tilde H,k}^{Iw}(\epsilon\cdot(1\times\omega^{2-k})$, respectively.\\
        Moreover, as before, define $v^{(\epsilon)}_k$ to be the sequence of slopes of the $T_p$ operator on the space $S_{\tilde H,k}^{ur}(\epsilon_1)$ (note we don't write $ur$ in the superscript for $v^{(\epsilon)}_k$ for simplicity as we will be using that sequence the most). 
    \end{notation}

    \begin{theorem}
        Fix a prime $p\geq11$, level $\Gamma_0(N)$, $a\in\{1,\ldots,p-5\}$ even and let $b\in\{0,1,\ldots,p-2\}$. Then for any Galois representations $\bar r:\Gal(\bar\QQ/\QQ)\to\GL_2(\FF)$ that is absolutely irreducible but when restricted to the intertia group of the form 
            \begin{equation}
                \bar r_{I_{\QQ_p}}\simeq\begin{pmatrix}
                    \omega_1^{a+b+1} & *\neq0\\
                    0 & \omega_1^b
                \end{pmatrix}
            \end{equation}
        the sequence $s^{(\epsilon)}_k$ in \cref{varslopeconj} equals to the sequence $v^{(\epsilon)}_k$.
    \end{theorem}
We proceed by induction on weight $k$. We prove the claim for a fixed $\bar r$ while letting $\epsilon$ vary. It suffices to prove for the case when $\tilde H$ is primitive, hence we now assume $m(\tilde H)=1$. Note that $\bar r|_{I_{\QQ_p}}\simeq\begin{pmatrix}\omega_1^{a+b+1} & *\neq 0\\ 0 & \omega_1^b\end{pmatrix}$. We write a proof below with reference to the necessary lemmas that will be proved below. More details for each case can be found in the corresponding sections.\\

Recall the notations from the previous section. We write $k=yp^x+(z-1)p^{x-1}+t+1$. We will use the term ``ghost coordinates'' to mean the points on the cartesian plane consisting of $(n,v_p(g_n))$. Recall that we denote $d_k^{ur}(\epsilon_1)$ and $d_k^{Iw}(\tilde\epsilon_1)$ be the dimension of the spaces $S^{ur}_{\tilde H,k}(\epsilon_1)$ and $S_{\tilde H,k}^{Iw}(\tilde\epsilon_1)$ respectively.\\

\begin{proof}
We now start by proving the base step. For $k\leq p+1$, using the ghost series, it is follows from the ghost series that all the ghost coordinates up to $x$-coordinate $d_{k}^{ur}(\epsilon_1)$ are $p$-adic units, hence we get multiple 0s. For the inductive step, assume that for all weights $k'$ smaller than $k$ and $\epsilon$ with $k_{\epsilon}\equiv k\mod p-1$, we have $s^{(\epsilon)}_{k'}=v^{(\epsilon)}_{k'}$.\\

As the algorithm is defined, we split into three cases.
\begin{enumerate}
\item $y+z\leq p-1$: We get that
    \begin{equation}
        v^{(\epsilon)}_k=\sigma(v^{(\epsilon),\dagger}_{e+1},d_k^{ur}(\epsilon_1))
    \end{equation}
since the ghost coordinates agree up to $x$ coordinate $d_{e+1}^{(\epsilon)}$ by \cref{first e+1} and the point with $x$ axis $d_k^{(\epsilon)}(\epsilon_1)$ lies on the Newton polygon $\NP(G_{\bar r}^{(\epsilon)}(w_{e+1},t))$ by \cref{smallweightoconv}.\\
Now, by \cref{atkin lehnerfor e+1}
\begin{equation}
    v_{e+1}^{(\epsilon),Iw}[d-i+1]=e-v_{e+1}^{(\epsilon'')}[i],
\end{equation}
and by \cref{atkinlehner k_2}, 
\begin{equation}
    v_p(g_n^{(\epsilon'')}(w_{k_2}))=v_p(g_n^{(\epsilon'')}(w_{e+1})),
\end{equation}
and hence from \cref{smallweightoconv},
\begin{equation}
    s_{k_2}^{(\epsilon'')}=\sigma(t_{e+1}^{(\epsilon'')},d_{k_2}^{ur}(\epsilon''_1)).
\end{equation}
From the facts above,
\begin{equation}
    v^{(\epsilon)}_{e+1}=v^{(\epsilon)}_{k_1}\cup \sigma(e-v^{(\epsilon'')}_{k_2},d_{k_1}^{ur}(\epsilon_1)+1,d_{e+1}^{Iw}(\epsilon\cdot(1\times\omega^{1-e}))).
\end{equation}
Finally, by \cref{final}, we get
    \begin{equation}
        v^{(\epsilon)}_k=\sigma(v^{(\epsilon)}_{k_1}\cup \sigma(e-v^{(\epsilon'')}_{k_2},d_{k_1}^{ur}(\epsilon_1)+1,d_{e+1}^{Iw}(\epsilon\cdot(1\times\omega^{1-e})))\cup v^{(\epsilon')}_{2B-k},d_k^{ur}(\epsilon_1))
    \end{equation}
\item $y<p-1<y+z$:
By \cref{firstk_1var}, \cref{increase},
\begin{equation}
    \sigma(v^{(\epsilon)}_k,d_{k_1}^{ur}(\epsilon_1))=v^{(\epsilon)}_{k_1}.
\end{equation}
Moreover, by \cref{smallweightoconv}, and \cref{midtermscase2}, \cref{firstk_1var} along with \cref{pstab}, we get
    \begin{equation}
        \begin{split}
            \sigma(v^{(\epsilon)}_k,d_{k_1}^{Iw}(\epsilon\cdot(1\times\omega^{2-k}))=v_1\cup \sigma(\min(v_2,\frac{k_1-2}{2}),d_{k_1}^{ur}(\epsilon_1)+1,\lfloor\frac{1}{2}d_{k_1}^{Iw}(\epsilon\cdot(1\times\omega^{2-k}))\rfloor)\\
            \cup [\frac{k_1-2}{2}]\cup (k_1-1-\sigma(\min(v_2,\frac{k_1-2}{2}),d_{k_1}^{ur}(\epsilon_1)+1,\lfloor\frac{1}{2}d_{k_1}^{Iw}(\epsilon\cdot(1\times\omega^{2-k}))\rfloor))\cup (k_1-v_1)
        \end{split}
    \end{equation}
when $d_{k_1}^{Iw}(\epsilon\cdot(1\times\omega^{2-k}))$ is odd, and 
    \begin{equation}
        \begin{split}
            \sigma(v^{(\epsilon)}_k,d_{k_1}^{Iw}(\epsilon\cdot(1\times\omega^{2-k}))=v_1\cup \sigma(\min(v_2,\frac{k_1-2}{2}),d_{k_1}^{ur}(\epsilon_1)+1,\lfloor\frac{1}{2}d_{k_1}^{Iw}(\epsilon\cdot(1\times\omega^{2-k}))\rfloor)\\
            \cup (k_1-1-\sigma(\min(v_2,\frac{k_1-2}{2}),d_{k_1}^{ur}(\epsilon_1)+1,\lfloor\frac{1}{2}d_{k_1}^{Iw}(\epsilon\cdot(1\times\omega^{2-k}))\rfloor))\cup (k_1-v_1)
        \end{split}
    \end{equation}
when it is even, where $v_1,v_2$ are as in \cref{varslopeconj}.\\
Finally, by \cref{final}, we get $v^{(\epsilon)}_k=V\cup\sigma(v^{(\epsilon')}_{2B-k},d_{k_1}^{ur}(\epsilon_1)+1,d_k^{ur}(\epsilon_1))$
\item $y=p-1$: By \cref{firstk_1var}(which also holds in case 3) and \cref{increase}
    \begin{equation}
        \sigma(v^{(\epsilon)}_k,d_{k_1}^{ur}(\epsilon_1))=v^{(\epsilon)}_{k_1}.
    \end{equation}
Moreover, by Lemma~\ref{midtermcase3}, Theorem~\ref{smallweightoconv}, and \cref{pstab} we get
\begin{equation}
    \begin{split}
        \sigma(v^{(\epsilon)}_k,d_{k_1}^{Iw}(\epsilon\cdot(1\times\omega^{2-k}))=v_1\cup \sigma(\min(v_2+1,\frac{k_1-2}{2}),d_{k_1}^{ur}(\epsilon_1)+1,\lfloor\frac{1}{2}d_{k_1}^{Iw}(\epsilon\cdot(1\times\omega^{2-k}))\rfloor)\\
        \cup [\frac{k_1-2}{2}]\cup (k_1-1-\sigma(\min(v_2+1,\frac{k_1-2}{2}),d_{k_1}^{ur}(\epsilon_1)+1,\lfloor\frac{1}{2}d_{k_1}^{Iw}(\epsilon\cdot(1\times\omega^{2-k}))\rfloor))\cup (k_1-v_1).
    \end{split}       
\end{equation}
when $d_{k_1}^{Iw}(\epsilon\cdot(1\times\omega^{2-k}))$ is odd, and 
\begin{equation}
    \begin{split}
        \sigma(v^{(\epsilon)}_k,d_{k_1}^{Iw}(\epsilon\cdot(1\times\omega^{2-k}))=v_1\cup \sigma(\min(v_2+1,\frac{k_1-2}{2}),d_{k_1}^{ur}(\epsilon_1)+1,\lfloor\frac{1}{2}d_{k_1}^{Iw}(\epsilon\cdot(1\times\omega^{2-k}))\rfloor)\\
        \cup (k_1-1-\sigma(\min(v_2+1,\frac{k_1-2}{2}),d_{k_1}^{ur}(\epsilon_1)+1,\lfloor\frac{1}{2}d_{k_1}^{Iw}(\epsilon\cdot(1\times\omega^{2-k}))\rfloor))\cup (k_1-v_1).
    \end{split}       
\end{equation}
when it is even, where $v_1$ and $v_2$ are as in Algorithm~\ref{varslopeconj}.\\
Finally, by Theorem~\ref{final}, we get $v^{(\epsilon)}_k=V\cup\sigma(v^{(\epsilon')}_{2B-k},d_{k_1}^{ur}(\epsilon_1)+1,d_k^{ur}(\epsilon_1))$.
    \end{enumerate}
Hence, by the induction hypothesis that $v_k=t_k$ for all $k>2$ and when $k_1=2$ the equation $v_2=\sigma(t_2,d_2(\epsilon_1))$ shows that $v_k=t_k$ for the inductive step too.
\end{proof}

\subsection{Main lemmas}

\begin{lemma}\label{GuMaz}
        Let $k\equiv a+2s_\epsilon+2\mod p-1$ and $C^{(\epsilon)}_{\bar r}(w_k,t)$ be the characteristic polynomial of the $U_p$ operator on $S_{\tilde H,k}^{Iw}(\tilde \epsilon_1)$. Then there is a positive integer $M$ such that for all $N>M$, $NP(C^{(\epsilon)}_{\bar r}(w_{k+p^N(p-1)},t)$ contains $NP(C^{(\epsilon)}_{\bar r}(w_k,t))$ below the part of slope $\frac{k-2}{2}$. 
    \end{lemma}
    \begin{proof}
        This follows from Theorem 8.10 of \cite{liu2023slopes}.
    \end{proof}
This will be used to prove the following:

  \begin{theorem}\label{smallweightoconv}
    Let $|k_1|\leq p^x$ and $|k_1|<k<p^{x+1}$ satisfy $k\equiv k_1\mod p^{x-1}.$ Then, if $m^{(\epsilon)}_{k_1}(d_k^{ur}(\epsilon_1))$ is zero, $(d_k^{ur}(\epsilon_1),v_p(g^{(\epsilon)}_{d_k^{ur}(\epsilon_1)}(w_{k_1})))$ lies on the Newton polygon  $\NP(G^{(\epsilon)}_{\bar r}(w_{k_1},t))$.
\end{theorem} 

\begin{proof}
    First, assume the contrary that the point with $x$-coordinate $d_k^{ur}(\epsilon_1)$ is above the segment of the Newton polygon. Then by \cref{NPpts} we can take the maximal near-Steinberg range defined for $k_m$. We abbreviate $l=L_{w_{k_1},k_m}$. By \cref{NPpts} and \cref{boundondifDelta} we have 
        \begin{equation}
            x-1\geq v_p(w_{k_1}-w_{k_m})\geq \Delta_{k_m,l}-\Delta_{k_m,l-1}>l-1,
        \end{equation}
    hence we get $l<x$. Moreover, if we let $k_t$ is the smallest number larger than $k$ that satisfies $v_p(k_1-k_t)\geq x$, for $k'<k_t$, 
        \begin{equation}\label{compkk1}
            v_p(w_{k'}-w_k)=v_p(w_{k'}-w_{k_1}).
        \end{equation} 
    By the definition of $k_t$, we have that $k_t\equiv k\mod p-1$. Note that $d^{ur}_{k_t}(\epsilon_1)>\frac{1}{2}d_{k_m}^{Iw}+l$ in all cases. Also due to the definition of the variables, for $n<d_{k_t}^{ur}(\epsilon_1)$ and $m$ such that $m^{(\epsilon)}_m(n)\neq0$, 
        \begin{equation}\label{compkk'}
            v_p(w_{k+p^M(p-1)}-w_m)=v_p(w_{e+1}-w_m).
        \end{equation}
    Note that $M$ is given by \cref{GuMaz}. Now due to \cref{compkk'} and \cref{compkk1}, we get that the ghost coordinates of the ghost series associated to $G^{(\epsilon)}_{\bar r}(w_{k+p^M(p-1)})$ and $G^{(\epsilon)}_{\bar r}(w_{e+1})$ have coefficients of $t^n$ with the same $p$-adic valuation for $d_{k_m}^{Iw}-l\leq n \leq d_{k_t}^{ur}$. By our assumption that the point $(d_k^{ur}(\epsilon_1,g^{(\epsilon)}_{d_k^{ur}(\epsilon_1)}(w_{k_1}))$ is above the Newton polygon contradicts the fact that $\NP(G^{(\epsilon)}_{\bar r}(w_k))$ appears in $\NP(G^{(\epsilon)}_{\bar r}(w_{k+p^M(p-1)}))$ up to $x$ coordinate $d_k^{ur}(\epsilon)$. 
\end{proof}

\begin{lemma}\label{increase}
        Let $k_1,k$ be as in the above three cases. Then we have 
        \begin{equation}
            \sigma(v^{(\epsilon)}_k,d_{k_1}^{ur}(\epsilon_1))=v^{(\epsilon)}_{k_1}.
        \end{equation}
    \end{lemma}

    \begin{proof}
        We note that this lemma is proved along with the induction steps of the proof of the theorem. In other words, we prove the lemma for a given $k$ where we are assuming the main theorem holds for $k'<k$. This is possible as when we prove the main theorem for a given $k$, we only need this lemma for the specified $k$.\\
        For $k_1,k$ in case 1, this follows from \S\ref{case 1}. Now assume that $k$ is in case 2 or 3. We divide into cases.
            \begin{enumerate}
                \item when $z\neq p$: $k_2=(y-1)p^x+(z+1)p^{x-1}+t+1$. Hence by the induction hypothesis, we get 
                    \begin{equation}
                        \sigma(v^{(\epsilon)}_{k_2},d_{k_1}^{ur}(\epsilon_1))=v^{(\epsilon)}_{k_1}.
                    \end{equation}
                This implies that 
                    \begin{equation}
                        \sigma(v^{(\epsilon)}_k,d_{k_1}^{ur}(\epsilon_1))=v^{(\epsilon)}_{k_1}
                    \end{equation} 
                since we have proved \cref{firstk_1var} and \cref{midtermscase2}, \cref{midtermcase3}.
                \item when $z=p$ and $y\neq p-1$: $k_2=yp^x+t+1$ and $k_1=yp^{x-1}+t+1$. Then $k_2$ falls in case 1, and hence we have
                    \begin{equation}
                        \sigma(v^{(\epsilon)}_{k_2},d_{k_1}^{ur}(\epsilon_1))=v^{(\epsilon)}_{k_1}.
                    \end{equation}
                again implying that 
                    \begin{equation}
                        \sigma(v^{(\epsilon)}_k,d_{k_1}^{ur}(\epsilon_1))=v^{(\epsilon)}_{k_1}.
                    \end{equation}
                \item when $z=p$ and $y=p-1$: $k_2=(p-1)p^x+t+1$ and $k_1=(p-1)p^{x-1}+t+1$. Using Theorem~\ref{smallweightoconv}, we get both $d_{k_1}^{ur}(\epsilon_1)$ and $d_{k_2}^{ur}(\epsilon_1)$ are vertices of the Newton polygon of $G^{(\epsilon)}_{\bar r}(w_{t+1})$, and the slopes appearing in $v^{(\epsilon)}_{k_1}$ and $v^{(\epsilon)}_{t+1}$ coincide outside the range $(d_{t+1}^{ur}(\epsilon_1),d_{t+1}^{Iw}(\tilde\epsilon_1)-d_{t+1}^{ur}(\epsilon_1))$ by \cref{firstk_1var}, and the same holds for $v^{(\epsilon)}_{k_1}$. For the slopes in the range $(d_{t+1}^{ur}(\epsilon_1),d_{t+1}^{Iw}(\tilde\epsilon_1)-d_{t+1}^{ur}(\epsilon_1))$, they follow the algorithm above, by the induction hypothesis, hence take values that are between $\max(v^{(\epsilon)}_{t+1})$ and $k_1-1-\max(v^{(\epsilon)}_{t+1})$. Hence we again have 
                    \begin{equation}
                        v^{(\epsilon)}_{k_2}[d_{k_1}^{ur}(\epsilon_1)+1]\geq v^{(\epsilon)}_{k_1}[d_{k_1}^{ur}(\epsilon_1)].
                    \end{equation}
                implying the result along with Lemma~\ref{firstk_1var} and Lemma~\ref{midtermcase3}.\qedhere
            \end{enumerate}
    \end{proof}

    \begin{remark}
        Note that Lemma~\ref{increase} logically depends on \S\ref{case 1} and Lemma~\ref{firstk_1var}, Lemma~\ref{midtermcase3}, and Lemma~\ref{midtermscase2} where Lemma~\ref{midtermscase2} is only dependent on Theorem~\ref{smallweightoconv}, hence this is not a circular logic. The order of the statements has been arranged in this fashion to maximize legibility. The following diagram shows the logical dependence.
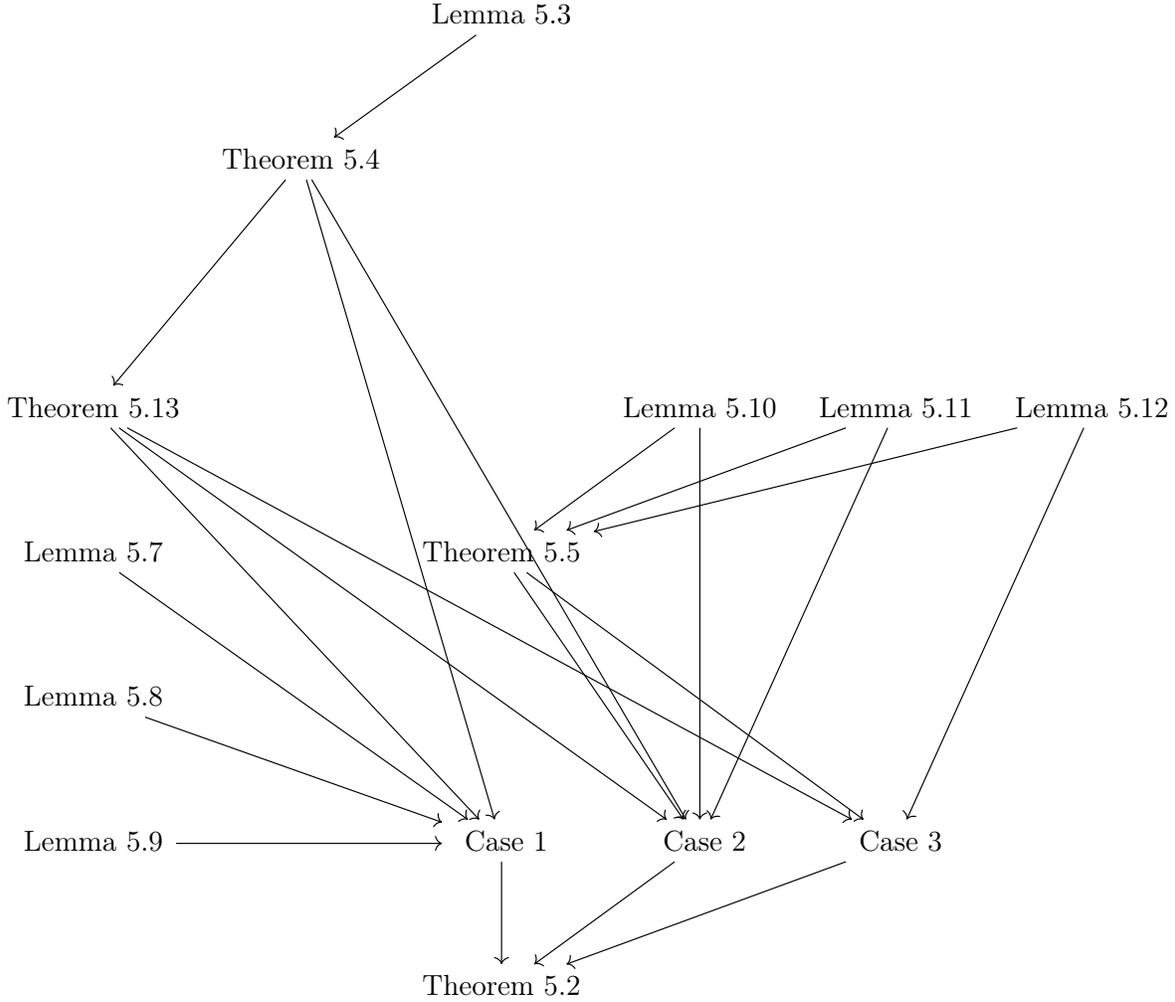
\[\begin{tikzcd}[column sep=tiny,row sep=huge]
	&& {\text{Lemma }5.3} \\
	& {\text{Theorem 5.4}} \\
	\\
	{\text{Theorem 5.13}} &&& {\text{Lemma 5.10}} & {\text{Lemma 5.11}} & {\text{Lemma 5.12}} \\
	{\text{Lemma 5.7}} && {\text{Theorem 5.5}} \\
	{\text{Lemma 5.8}} \\
	{\text{Lemma 5.9}} && {\text{ Case 1}} & {\text{ Case 2}} & {\text{ Case 3}} \\
	&& {\text{Theorem 5.2}}
	\arrow[from=1-3, to=2-2]
	\arrow[from=2-2, to=4-1]
	\arrow[from=4-4, to=5-3]
	\arrow[from=4-5, to=5-3]
	\arrow[from=4-6, to=5-3]
	\arrow[from=2-2, to=7-3]
	\arrow[from=4-1, to=7-3]
	\arrow[from=4-1, to=7-4]
	\arrow[from=4-1, to=7-5]
	\arrow[from=5-3, to=7-4]
	\arrow[from=5-3, to=7-5]
	\arrow[from=4-4, to=7-4]
	\arrow[from=4-5, to=7-4]
	\arrow[from=4-6, to=7-5]
	\arrow[from=2-2, to=7-4]
	\arrow[from=7-3, to=8-3]
	\arrow[from=7-4, to=8-3]
	\arrow[from=7-5, to=8-3]
	\arrow[from=6-1, to=7-3]
	\arrow[from=7-1, to=7-3]
	\arrow[from=5-1, to=7-3]
\end{tikzcd}\]
    \end{remark}

In the next three sections, we will prove the lemmas cited above along with elaboration on the arguments. To understand the details of the proof, we recommend the reader to read the outline first and the following sections after one has become familiar of the structure of the proof.
\subsection{When $y+z\leq p-1$}\label{case 1}
Consider the notation above, 
    \begin{equation*}
        \begin{split}
            & k_1=k-y(p-1)p^{x-1}=(y+z-1)p^{x-1}+t+1\\
        & k_2=(p-y-z)p^{x-1}+t+1.
        \end{split}
    \end{equation*} 
The Ghost series for $\bar r$ is given by 
    \begin{equation*}
        G^{(\epsilon)}_{\bar r}(w,t)=\sum_{n\geq0}g^{(\epsilon)}_n(w)(t)=1+\sum_{n\geq0}\prod_{l\equiv k\mod p-1}(w-w_l)^{m^{(\epsilon)}_n(l)}t^n
    \end{equation*}
We have the following lemma.
    \begin{lemma}\label{firstk_1}
        For $0\leq n\leq d_{k_1}^{ur}(\epsilon_1)$, 
            \begin{equation*}
                v_p(g^{(\epsilon)}_n(w_{k_1}))=v_p(g^{(\epsilon)}_n(w_k)),
            \end{equation*} 
        hence $(n,v_p(\prod_{l\equiv k\mod p-1}(w_{k_1}-w_l)^{m^{(\epsilon)}_n(l)})$ appears as the first $d_{k_1}^{ur}(\epsilon_1)+1$ points of the ghost coordinates of $k$.
    \end{lemma}

    \begin{proof}
        For each $l<k_1$, we have 
            \begin{equation*}
                v_p(w_k-w_l)=v_p(w_{k_1}-w_l)
            \end{equation*}
        since 
        \begin{equation*}
            k\equiv l\mod p^x\text{  and  }k\equiv l\mod p-1
        \end{equation*}
        is impossible as $k_1\leq p^x$. Also, for $n\leq d_{k_1}^{ur}(\epsilon_1)$, $d_l^{ur}<n<d_l^{Iw}-d_l^{ur}$ is necessary for $m^{(\epsilon)}_n(l)$ to be nonzero hence $l<k_1$. Thus 
            \begin{equation}
                v_p\bigg(\prod_{l\equiv k\mod p-1}(w_{k_1}-w_l)^{m^{(\epsilon)}_n(l)}\bigg)=v_p\bigg(\prod_{l\equiv k\mod p-1}(w_{k}-w_l)^{m^{(\epsilon)}_n(l)}\bigg)
            \end{equation}
        and the points are identical.
    \end{proof}

We have another lemma comparing the ghost coordinates for $k$ with another space of modular forms.
    \begin{lemma}\label{first e+1}
        For $n\leq d_{e+1}^{Iw}(\epsilon\cdot(1\times\omega^{1-e}))$,
            \begin{equation}
                v_p\big(g^{(\epsilon)}_n(w_{e+1})\big)=v_p\big(g^{(\epsilon)}_n(w_k)\big).
            \end{equation}
        In other words, the ghost coordinates of the space $S_{\tilde H,e+1}^{Iw}(\epsilon\cdot(1\times\omega^{1-e}))$ is identical to the first $d_{e+1}^{Iw}(\epsilon\cdot(1\times\omega^{1-e}))$ ghost coordinates of the ghost series for $k$.
    \end{lemma}

    \begin{proof}
        We prove this again by comparing the $p$-adic valuations of the ghost coefficients of the ghost series $G^{(\epsilon)}_{\bar r}(w,t)$ when $w=w_{e+1}$ and $w=w_k$. \\
        Note that $e+1=k-B+1=t+1\leq p^{x-1}$. Hence we again have 
        \begin{equation*}
            v_p(w_k-w_l)=v_p(w_{e+1}-w_l)
        \end{equation*}
        for $d^{ur}_l(\epsilon_1)\leq d_{e+1}^{Iw}(\epsilon\cdot(1\times\omega^{1-e}))$. If $m^{(\epsilon)}_l(n)>0$, then $d_l^{ur}(\epsilon_1)<n<d_l^{Iw}(\tilde\epsilon_1)-d_l^{ur}(\tilde\epsilon_1)$ and hence we get $v_p(w_k-w_l)=v_p(w_{e+1}-w_l)$ implying the desired result.
    \end{proof}
Hence, applying \cref{smallweightoconv} for $(e+1,k)$ and $(e+1,k_1)$(meaning that we are letting the $k_1,k$ in \cref{smallweightoconv} as the tuples specified), we get that 
    \begin{equation}
        \begin{split}
            \sigma(v^{(\epsilon),\dagger}_{e+1},d_{k_1}^{ur}(\epsilon_1))=v^{(\epsilon)}_{k_1}\\
        \sigma(v^{(\epsilon),\dagger}_{e+1},d_{k}^{ur}(\epsilon_1))=v^{(\epsilon)}_{k}
        \end{split}
    \end{equation}
and combining the two, we get 
    \begin{equation}\label{first k_1 appears firstcase}
        \sigma(v^{(\epsilon)}_{k},d_{k_1}^{ur}(\epsilon_1))=v^{(\epsilon)}_{k_1}.
    \end{equation}
Now we use the property of the ghost series being compatible with the Atkin-Lehner involution \cref{atkinlehner}. We have 
    \begin{equation}\label{atkin lehnerfor e+1}
        v^{(\epsilon),Iw}_{e+1}[i]+v^{(\epsilon''),Iw}_{e+1}[d-i+1]=e
    \end{equation}
where $s_{\epsilon''}=e-1-a-s_\epsilon$ and $d=\dim S_{\tilde H,e+1}^{Iw}(\epsilon\cdot(1\times\omega^{1-e}))$. From now on, we denote $\epsilon_1''$ for the corresponding first factor of $\epsilon''$ as we write $\epsilon_1$ for $\epsilon$ a character of $(\FF_p^\times)^2$.(Note that we will use the same notation for another character $\epsilon'$ later in this section)
 Now we state another lemma.

    \begin{lemma}\label{atkinlehner k_2}
        For $n\leq d_{k_2}^{ur}(\epsilon_1'')$, we have 
            \begin{equation}
                v_p(g^{(\epsilon'')}_n(w_{k_2}))=v_p(g^{(\epsilon'')}_n(w_{e+1})).
            \end{equation}
    \end{lemma}

    \begin{proof}
        This amounts to proving that the first $d_{k_2}^{ur}(\epsilon_1')$ terms in the ghost series $G^{(\epsilon'')}_{\bar r}(w,t)$ with $w=w_{k_2}$ and $w=w_{e+1}$ have the same $p$-adic valuation. This is true as $e+1=t+1$ and $k_2=(p-y-z)p^{x-1}+t+1$ and hence there is no $l$ such that $l\equiv k_2\mod p-1$ and $v_p(k_2-l)=x$. (Details omitted due to repetitive arguments.)
    \end{proof}

Hence we deduce that the first part of the algorithm gives us the first $d_{e+1}^{Iw}(\epsilon\cdot(1\times\omega^{1-e}))$ slopes of the space $S_{\tilde H,k}^{ur}(\epsilon_1)$.\\
\\
By Theorem~\ref{someNPpts}, we get $v^{(\epsilon),Iw}_{e+1}$ appears as the first terms of $v^{(\epsilon),\dagger}_{e+1}$. By Theorem~\ref{smallweightoconv}, we get that $v^{(\epsilon)}_k$ is a truncation of $ v^{(\epsilon),\dagger}_{e+1}$ up to the $d_k^{ur}(\epsilon_1)$'th term. Hence we get that $v^{(\epsilon),Iw}_{e+1}$ appears as the first $d_{1+e}^{Iw}(\epsilon\cdot(1\times\omega^{1-e}))$ slopes of $S_{\tilde H,k}^{ur}(\epsilon_1)$. Formally, this means
    \begin{equation}
        v^{(\epsilon),Iw}_{e+1}=\sigma(v^{(\epsilon)}_k,d_{e+1}^{Iw}(\tilde\epsilon))
    \end{equation}
Moreover, using Theorem~\ref{smallweightoconv} for $k_1$ and $k$, the first $d_{k_1}^{ur}(\epsilon_1)$ slopes of $S_{\tilde H,k_1}^{ur}(\epsilon_1)$ give the first $d_{k_1}^{ur}(\epsilon_1)$ slopes of $S_{\tilde H,k}^{ur}(\epsilon_1)$, i.e., 
    \begin{equation}
        \sigma(v^{(\epsilon)}_k,d_{k_1}^{ur}(\epsilon_1))=v^{(\epsilon)}_{k_1}.
    \end{equation}
The rest come from taking $e$ minus the ones from $S_{\tilde H,k_2}^{ur}(\epsilon'_1)$ with reverse order as we have 
    \begin{equation}
        v^{(\epsilon),Iw}_{e+1}[d-i+1]=e-v^{(\epsilon'')}_{k_2}[i]
    \end{equation}
from \cref{atkin lehnerfor e+1} and since THeorem~\ref{smallweightoconv} with Lemma~\ref{atkinlehner k_2} gives us that $t^{(\epsilon'')}_{k_2}$ equals $\sigma(t^{(\epsilon''),Iw}_{e+1},d_{k_2}^{ur}(\epsilon''_1))$.
The fact that the resulting sequence of the algorithm is increasing can be seen as follows. First, by the inequalities as mentioned before, we get that the sum of dimensions 
    \begin{equation}
        d_{k_1}^{ur}(\epsilon_1)+d_{k_2}^{ur}(\epsilon_1'')>d_{e+1}^{Iw}(\epsilon\cdot(1\times\omega^{1-e}))
    \end{equation} 
by \cref{dimform}. Also, we proved that the sequence $e-t^{(\epsilon'')}_{k_2}$  appears at the end of the sequence $t^{(\epsilon)}_k$ and $ t^{(\epsilon)}_{k_1}$ appears the first $d_{k_1}^{ur}(\epsilon_1)$ slopes, and the sum of the length of the two sequences is larger than the total length hence there exists an index $i$ such that 
    \begin{equation}
        t^{(\epsilon)}_{k_1}[i]=e-t^{(\epsilon'')}_{k_2}[d_{e+1}^{Iw}(\epsilon\cdot(1\times\omega^{1-e}))-i+1]
    \end{equation}
implying that they form an increasing sequence.
\subsection{When $y<p-1< y+z$}
Now, in this case, we proceed in a similar fashion in the first step. First, using the notation as before, we get $k_1=(y+z-p)p^{x-1}+t+1$ and $k_2=(y-1)p^x+zp^{x-1}+t+1$ and $e+1=k_1$. We prove similar lemmas as in the previous section.
    \begin{lemma}\label{firstk_1var}
        The coordinates $(n,v_p(g^{(\epsilon)}_n(w_{k_1})))$ up to $n=d_{k_1}^{Iw}(\tilde\epsilon_1)$ coincide with the points with $x$ coordinate at most $d_{k_1}^{Iw}(\tilde\epsilon_1)$ among the points $(n,v_p(g^{(\epsilon)}_n(w_k)))$ for $n$ not in the range $(d_{k_1}^{ur}(\epsilon_1),d_{k_1}^{Iw}(\tilde\epsilon_1)-d_{k_1}^{ur}(\epsilon_1))$ (including $n=0$).
    \end{lemma}

    \begin{proof}
        First, for each $l<k_1$, we have $v_p(w_k-w_l)=v_p(w_{k_1}-w_l)$ since satisfying $k\equiv l\mod p^x$ and $k\equiv l\mod p-1$ is impossible as $y<p-1$. Also, for $n\leq d_{k_1}^{ur}$, $m^{(\epsilon)}_n(l)$ being nonzero implies $d_l^{ur}(\epsilon_1)<n<d_l^{Iw}(\tilde\epsilon_1)-d_l^{ur}(
        \epsilon
        _1)$ hence $l<k_1$. Thus 
            \begin{equation}
                v_p\bigg(\prod_{l\equiv k\mod p-1}(w_{k_1}-w_l)^{m^{(\epsilon)}_n(l)}\bigg)=v_p\bigg(\prod_{l\equiv k\mod p-1}(w_{k}-w_l)^{m^{(\epsilon)}_n(l)}\bigg)
            \end{equation}
        and the ghost coordinates are identical. \\
        Now consider the case when $d_{k_1}^{Iw}(\tilde\epsilon_1)-d_{k_1}^{ur}(\epsilon_1)<n<d_{k_1}^{Iw}(\tilde\epsilon_1)$. Then $m^{(\epsilon)}_n(l)$ is nonzero if and only if $d_l^{ur}(\epsilon_1)<n<d_l^{Iw}(\tilde\epsilon_1)-d_l^{ur}(\epsilon_1)$, hence we get $k_1<l<(p+1)k_1$. Hence in that interval, all values $v_p(w_l-w_{k_1})$ are equal to $v_p(w_k-w_l)$, hence we get the desired result.
    \end{proof}

    \begin{lemma}\label{midtermscase2}
        The ghost coordinates up to $x$ coordinate $d_{k_2}^{ur}(\epsilon_1)$ that appear between $d_{k_1}^{ur}(\epsilon_1)$ and $d_{k_1}^{Iw}(\tilde\epsilon_1)-d_{k_1}^{ur}(\epsilon_1)$ are given by the ghost coordinates of weight $k_2$. Formally, 
            \begin{equation}
                v_p(g^{(\epsilon)}_n(w_k))=v_p(g^{(\epsilon)}_n(w_{k_2}))\text{ for all }d_{k_1}^{ur}(\epsilon)\leq n\leq d_{k_2}^{ur}(\epsilon_1)
            \end{equation}
    \end{lemma}

    \begin{proof}
        Fist, using dimension formulas, we get that $d_{k_2}^{ur}(\epsilon_1)$ is smaller than $d_{k_1}^{Iw}(\tilde\epsilon_1)-d_{k_1}^{ur}(\epsilon_1)$ and at least $\frac{1}{2}d_{k_1}^{Iw}(\tilde\epsilon_1)$.  Using the same arguments as before, we have that $v_p(w_k-w_l)=v_p(w_{k_2}-w_l)$ for $l<k_2$ and hence we get the desired result.
    \end{proof}

Now, by Lemma~\ref{increase}, Lemma~\ref{firstk_1var} and \cref{pstab} for $k_1$, we get 
    \begin{equation}
        v^{(\epsilon)}_k[i]+v^{(\epsilon)}_k[d_{k_1}^{Iw}(\tilde\epsilon_1)-i+1]=e \text{ for all }i\leq d_{k_1}^{ur}(\epsilon_1).
    \end{equation}
Now, by \cref{ghostdual} for the slopes in the range $[d_{k_1}^{ur}(\epsilon_1),d_{k_1}^{Iw}(\tilde\epsilon_1)-d_{k_1}^{ur}(\epsilon_1)]$, gives us that the slopes of the segments between $d_{k_1}^{ur}(\epsilon_1)$ and $d_{k_1}^{Iw}(\tilde\epsilon_1)-d_{k_1}^{ur}(\epsilon_1)$ satisfy 
    \begin{equation}
        v^{(\epsilon)}_k[d_{k_1}^{ur}(\epsilon_1)+i]+v^{(\epsilon)}_k[d_{k_1}^{Iw}(\tilde\epsilon_1)-d_{k_1}^{ur}(\epsilon_1)-i+1]=e-1,
    \end{equation}
and since we have \cref{midtermscase2}, we deduce that 
    \begin{equation}
        v^{(\epsilon)}_k[d_{k_1}^{ur}(\epsilon_1)+i]=\min(v^{(\epsilon)}_{k_2}[d_{k_1}^{ur}(\epsilon_1)+i],e_2)
    \end{equation}
for $i<\frac{1}{2}d_{k_1}^{Iw}(\tilde\epsilon_1)-d_{k_1}^{ur}(\epsilon_1)$ and 
    \begin{equation}\label{ghostdualforK2}
        v^{(\epsilon)}_k[d_{k_1}^{Iw}(\tilde\epsilon_1)-d_{k_1}^{ur}(\epsilon_1)-i-1]=e_2-\min(v^{(\epsilon)}_{k_2}[d_{k_1}^{ur}(\epsilon_1)+i],e_2).
    \end{equation}
Hence, we get
    \begin{equation}
        \sigma(t^{(\epsilon)}_k,d_{k_1}^{Iw})=\sigma(v^{(\epsilon)}_k,d_{k_1}^{Iw})
    \end{equation}
\subsection{When $y=p-1$}

We proceed in a similar fashion for the first step. First, using the notation as before, we get $k_1=(z-1)p^{x-1}+t+1$, $k_2=(p-2)p^x+zp^{x-1}+t+1$, and $e+1=k_1$.\\
Lemma~\ref{firstk_1var} holds exactly same in this case and explains the appearance of $v_1$ and $e-v_1$. Now we explain the construction of the sequence $w$. First, we have an analogue of  Lemma~\ref{midtermscase2}.

    \begin{lemma}\label{midtermcase3}
        The ghost coordinates for weight $k$ in the interval $n\in[d_{k_1}^{ur}(\epsilon_1),d_{k_1}^{Iw}(\tilde\epsilon_1)-d_{k_1}^{ur}(\epsilon_1)]\cap[0,d_{k_2}^{ur}(\epsilon_1)]$ are given by the ghost coordinates of $k_2$ with $n-d_{k_1}^{ur}(\epsilon_1)$ added. Formally, 
            \begin{equation*}
                \begin{split}
                    v_p(g^{(\epsilon)}_n(w_k))=v_p(g^{(\epsilon)}_n(w_{k_2}))+\min(n-d_{k_1}^{ur}(\epsilon_1),d_{k_1}^{Iw}(\tilde\epsilon_1)-d_{k_1}^{ur}(\epsilon_1)-n)\\
                \text{ for all }n\in[d_{k_1}^{ur}(\epsilon_1),d_{k_1}^{Iw}(\tilde\epsilon_1)-d_{k_1}^{ur}(\epsilon_1)]\cap[0,d_{k_2}^{ur}(\epsilon_1)]
                \end{split}
            \end{equation*}
    \end{lemma}

    \begin{proof}
        Looking at each of the ghost coordinate, we get $\prod_{l\equiv k\mod p-1}(w_k-w_l)^{m^{(\epsilon)}_n(l)}$. For every $l$ except $l=t+1$, we get $v_p(w_k-w_l)=v_p(w_{k_2}-w_l)$ and $v_p(w_k-w_{k_1})=v_p(w_{k_2}-w_{k_1})+1$. Hence we get the desired result. 
    \end{proof}

Now, as before, by \cref{increase}, \cref{firstk_1var}, and \cref{pstab} for $k_1$, we get 
    \begin{equation}
        v^{(\epsilon)}_k[i]+v^{(\epsilon)}_k[d_{k_1}^{Iw}(\tilde\epsilon_1)-i+1]=e \text{ for all }i\leq d_{k_1}^{ur}(\epsilon_1).
    \end{equation}
Now, by \cref{ghostdual} for the slopes in the range $[d_{k_1}^{ur}(\epsilon_1),d_{k_1}^{Iw}(\tilde\epsilon_1)-d_{k_1}^{ur}(\epsilon_1)]$, gives us that the slopes of the segments between $d_{k_1}^{ur}(\epsilon_1)$ and $d_{k_1}^{Iw}(\tilde\epsilon_1)-d_{k_1}^{ur}(\epsilon_1)$ satisfy 
    \begin{equation}
        v^{(\epsilon)}_k[d_{k_1}^{ur}(\epsilon_1)+i]+v^{(\epsilon)}_k[d_{k_1}^{Iw}(\tilde\epsilon_1)-d_{k_1}^{ur}(\epsilon_1)-i+1]=e-1,
    \end{equation}
and since we have \cref{midtermcase3}, we deduce that 
    \begin{equation}
        v^{(\epsilon)}_k[d_{k_1}^{ur}(\epsilon_1)+i]=\min(v^{(\epsilon)}_{k_2}[d_{k_1}^{ur}(\epsilon_1)+i]+1,e_2)
    \end{equation}
for $d_{k_1}^{ur}(\epsilon_1)<i<\frac{1}{2}d_{k_1}^{Iw}(\tilde\epsilon_1)-d_{k_1}^{ur}(\epsilon_1)$ and 
    \begin{equation}
        v^{(\epsilon)}_k[d_{k_1}^{Iw}(\tilde\epsilon_1)-d_{k_1}^{ur}(\epsilon_1)-i-1]=e_2-\min(v^{(\epsilon)}_{k_2}[d_{k_1}^{ur}(\epsilon_1)+i]+1,e_2).
    \end{equation}
This shows that the sequence $t^{(\epsilon)}_k$ we obtain from the algorithm coincides with $v^{(\epsilon)}_k$.
\subsection{The final part of the sequence}
Now we prove the necessary lemmas for the final step of the algorithm adding $e+v^{(\epsilon')}_{2B-k}$ (note $s_{\epsilon'}=s_\epsilon-e$). This can be proved simultaneously for all three cases of the algorithm.          
    \begin{theorem}\label{final}
        The slopes of the space $S_{\tilde H,2B-k}^{ur}(\epsilon_1\cdot\omega^e)$  with $e$ gives the remaining slopes of $S_{\tilde H,k}^{ur}(\epsilon_1)$. Formally,
            \begin{equation}
                v^{(\epsilon)}_k[d_{e+1}^{Iw}(\epsilon_1)+i]=e+v^{(\epsilon')}_{2B-k}[i]
            \end{equation}
    \end{theorem}

    \begin{proof}
        Using \cref{smallweightoconv}, we get 
            \begin{equation}
                v^{(\epsilon)}_{e+1})[i]=v^{(\epsilon)}_k[i]\text{ for all }i\in(d_{e+1}^{Iw}(\epsilon\cdot(1\times\omega^{1-e})),d_k^{ur}(\epsilon_1)].
            \end{equation}
        On the other hand, from \cref{ghostprops} \cref{theta}, we get that 
            \begin{equation}
                v^{(\epsilon)}_{e+1}[d_{e+1}^{Iw}(\epsilon\cdot(1\times\omega^{1-e}))+i]=e+v^{\dagger,(\epsilon')}_{1-e}[i]
            \end{equation}
        It remains to prove that the sufficiently small slopes of $S_{\tilde H,1-e}^\dagger(\epsilon'\cdot(1\times\omega^{1+e}))$ coincides with the slopes of $S_{\tilde H,2B-k}^{ur}(\epsilon'_1)$. Note that $2B-k-(1-e)=B-1$ is divisible by $p^{x-1}$, and the ghost coordinates are given by product of $w_{2B-k}-w_m$ or $w_{1-e}-w_m$ for $m<2B-k$ and the $p$-adic valuations are the same. Using Theorem~\ref{smallweightoconv} for $1-e$ and $2B-k$, we also get
            \begin{equation}
                v^{\dagger,(\epsilon)}_{1-e}[i]=v^{(\epsilon')}_{2B-k}[i]\text{ for all }i\leq d_{2B-k}^{ur}(\epsilon').
            \end{equation}
        Note that the dimensions  $d_{e+1}^{Iw}(\tilde\epsilon_1)$ and $d_{2B-k}^{ur}(\epsilon_1\cdot\omega^{e})$ add up to a larger value than $d_k^{ur}(\epsilon_1)$ by Proposition~\ref{dimform} thus we get that the algorithm gives the slopes of $S_{\tilde H,k}^{ur}(\epsilon_1)$.
    \end{proof}

These lemmas in total finally prove that the variant of the algorithm of Buzzard coincides with the first $d_k^{ur}$ terms of the slopes of the Newton polygon of the Ghost series as in the first part of this section, hence gives the slopes of the space of modular forms of weight $k$ localized at a suitable twist of the Galois representation given.

\appendix

\section{Slope conjecture of Kevin Buzzard}

In this section, we review Buzzard's slope conjecture from \cite{buzconj} which suggests an algorithm that outputs an infinite sequence of slopes of modular forms of fixed weight. None of this is original to the auther and is taken from Buzzard's paper \cite{buzconj}. We first define $\Gamma_0(N)$ regularity. Let $k_p=\frac{p+3}{2}$ if $p>2$ and $k_2=4$.
    \begin{definition}[$\Gamma_0(N)$-regularity]\label{gam0reg}
        If $p>2$, then we say that $p$ is $\Gamma_0(N)$-regular if the eigenvalues of $T_p$ acting on $S_k(\Gamma_0(N))$ are all $p$-adic units for all even integers $2\leq k\leq k_p$.\\
        If $p=2$, We say $2$ is \emph{$\Gamma_0(N)$-regular} if
        \begin{enumerate}
            \item The eigenvalues of $T_2$ on $S_2(\Gamma_0(N))$ are 2-adic units.
            \item There are exactly $\dim(S_2(\Gamma_0(2N)))-\dim(S_2(\Gamma_0(N)))$ eigenvalues of $T_2$ on $S_4(\Gamma_0(N))$ which are 2-adic units and all others have 2-adic valuation 1.
        \end{enumerate}        
    \end{definition}
Now we assume for the rest of the section that $p>3$ (moreover, we will later assume that $p\geq 11$ as we will be relating the recent proof of the ghost conjecture in \cite{liu2023slopes} and their constraint on $p$ is at least 11. Then any continuous odd irreducible Galois representation $\rho:\Gal(\bar\QQ/\QQ)\to\GL_2(\bar\FF_p)$ with determinant equal to an integer power of the cyclotomic character has a twist coming from a characteristic zero form of weight at most $p+1$, level equal to the conductor of $\rho$ and trivial character. Moreover, the eigenvalues of $T_p$ being $p$-adic units can be determined from the local behaviour of $\rho$ at $p$. Finally if there is a mod $p$ eigenform of level $N$ and weight $k$ with $k_p<k \leq p+1$ which is in the kernel of $T_p$, there is another such form of weight $p+3-k$. Hence we have the following lemma:
    \begin{lemma}\cite{buzconj}\label{gam0galois}
        $p>3$ is $\Gamma_0(N)$ regular if and only if any irreducible modular Galois representation $\rho\colon\Gal(\bar \QQ/\QQ)\to\GL_2(\bar\FF_p)$ with conductor dividing $N$ and determinant a power of the mod $p$ cyclotomic character is necessarily reducible when restricted to a decomposition group at $p$.
    \end{lemma}

Now we state Buzzard's conjecture with his algorithm.
    \begin{conjecture}\label{slopeconj}
        Assume that $p$ is $\Gamma_0(N)$-regular. Then the sequences $s_2,s_4,\ldots$ of integers are precisely the sequences $v_2,v_4,\ldots$ of $p$-adic valuations of $T_p$ acting on $S_k(\Gamma_0(N))$.
    \end{conjecture}

We use the following notation from Buzzard's paper.

    \begin{notation}
        \begin{itemize}\leavevmode
            \item We write a finite sequence as $s=[a_1,\ldots,a_n]$, denote $l(s)$ as its length, $s[i]$ as $a_i$. 
            \item For two sequences $a,b$, we write $a\cup b$ as the length $l(a)+l(b)$ sequence as $a$ followed by $b$. 
            \item If $l(a)=l(b)$, then $\min(a,b)$ is given by pointwise minimum.  
            \item For $n,r\geq 0,$ let $\kappa(n,r)$ to be the constant sequence of length $n$, value $r$. 
            \item If $v$ is a sequence, let $v+e$ be pointwise adding $e$ and $e-v$ be pointwise subtracting from $e$ with order reversed. 
            \item If $v$ has lenght at least $\delta$, then $\sigma(v,\delta)$ is the truncation up to $\delta$, and if $1\leq\delta_1,\delta_2\leq l(v)$, $\sigma(v,\delta_1,\delta_2)$ is the sequence cut from $\delta_1$ to $\delta_2$ (endpoints included).
            \item For $k\in\ZZ$, write $d(k)$ for the dimension of $S_k(\Gamma_0(N))$.
            \item Write $d_p(k)$ for the dimension of $S_k(\Gamma_0(Np))$.
            \item Where $\epsilon$ a Dirichelt character of level $p$, write $d_{p,\epsilon}(k)$ for the dimension of $S_k(\Gamma_0(N)\cap\Gamma_1(p),\epsilon)$.
        \end{itemize}
    \end{notation}

Now we can define the algorithm.

    \begin{algorithm}[{Buzzard's Slope algorithm}]\label{slopealg}
        We start defining a sequence $t_k$. It will turn out $s_2=\kappa(d(2),0)$ and $s_k=t_k$ for $k>2.$ For $4\leq k\leq p+1$, let $t_k=\kappa(d(k),0)$ and $t_2=\kappa(d_p(2)-d(2),0)$. Set $k_{min}=p+3$. \\
        Now, assume that $k\geq k_{min}$ is even and we have $t_l$ for all even $l<k$. We now define $t_k$ depending on three parameters $x,y,z$. \\
        $x$ is defined as the unique positive integer such that $$p^x<k-1\leq p^{x+1},$$ $y$ be the positive integer satisfying $$p^xy<k-1\leq p^x(y+1).$$ Set $$z=1+\bigg\lfloor\frac{k-2-p^xy}{p^{x-1}}\bigg\rfloor.$$ Then $1\leq z\leq p$. Let $m$ be the number of cusps of $X_0(N)$. We define a sequence $V$ which are the first few slopes of $t_k$. The algorithm used for $V$ will depend on $y,z$ on the following three cases: $y+z\leq p-1,y<p-1<y+z,y=p-1$.\\
        \begin{enumerate}
            \item When $y+z\leq p-1$: We let 
            \begin{equation*}
                \begin{split}
                    & k_1=k-y(p-1)p^{x-1} \\ & k_2=k-(y-1)(p-1)p^{x-1}-2(y+z-1)p^{x-1}.
                \end{split}
            \end{equation*}
            Set $v_1=t_{k_1}$ and $v_2=t_{k_2}$, and
            define 
            \[
            B=p^xy+p^{x-1}(z-1)+1,\ e=k-B,\ \epsilon=\chi^{1-B}
            \]
            where $\chi$ is any Dirichlet character of conductor $p$ and order $p-1$. Finally set 
            \begin{equation*}
                s=1+d_{p,\epsilon}(1+e).
            \end{equation*}
            If $l(v_1)\geq s-1$, then let $V_1=\sigma(v_1,s-1)$. \\
            Otherwise let $V_1=v_1\cup(e-\sigma(v_2,s-1-l(v_1)))$. \\
            Finally let $V=V_1\cup \kappa(m,e)$.\\
            \item When $y<p-1<y+z$: We set 
            \begin{equation*}
                \begin{split}
                    & k_1=k-((y+1)p^{x-1}(p-1))\\
                & k_2=k-p^{x-1}(p-1).
                \end{split}
            \end{equation*}
            We let $v_1=t_{k_1}$ and $v_2=t_{k_2}$, and
            define 
            \begin{equation*}
                B=(y+1)p^{x-1}(p-1)+1,\ e=k-B.
            \end{equation*}
            Finally set 
            \begin{equation*}
                s=1+d_p(1+e),\ s_2=\lfloor (s-1)/2\rfloor,\ e_2=\lfloor e/2\rfloor.
            \end{equation*}
            If $l(v_1)\geq s-1$, let $V_1=\sigma(v_1,s-1)$.\\
            Else if $s-1\leq 2l(v_1)<2(s-1)$, let $V_1=v_1\cup(e-\sigma(v_1,s-1-l(v_1)))$.\\
            Else then define $w=\sigma(v_2,l(v_1)+1,s_2)$.
            \begin{itemize}
                \item  If $s$ is even, let $V_1=v_1\cup w\cup [e_2]\cup(e-1-w)\cup(e-v_1)$, 
                \item if $s$ is odd, let $V_1=v_1\cup w\cup (e-1-w)\cup(e-v_1)$.
            \end{itemize}
            Finally if $e=1$, define $V=V_1\cup\kappa(m-1,1)$ and $V=V_1\cup\kappa(m,e)$ otherwise.
            \item When $y=p-1$: We let 
            \begin{equation*}
                \begin{split}
                    & k_1=k-p^x(p-1)\\
                & k_2=k-p^{x-1}(p-1).
                \end{split}
            \end{equation*}
            We set $v_1=t_{k_1}$ and $v_2=t_{k_2}$, and
            set 
            \begin{equation*}
                B=p^a(p-1), e=k-B.
            \end{equation*}
            Next, set $s=1+d_p(1+e)$ and $s_2$ and $e_2$ as above.
            If $l(v_1)\geq s-1$, then we set $V_1=\sigma(v_1,s-1-l(v_1))$.\\
            Else if $s-1\leq 2l(v_1)<2(s-1)$, let $V_1=v_1\cup(e-\sigma(v_1,s-1-l(v_1)))$.\\
            Else define $w_0=\sigma(v_2,l(v_1)+1,s_2)$ and $w=\min(w_0+1,\kappa(l(w_0),e_2))$ and 
            \begin{itemize}
                \item if $s$ is even $V_1=v_1\cup w\cup [e_2]\cup (e-1-w)\cup (e-v_1)$
                \item and if $s$ is odd $V_1=v_1\cup w\cup (e-1-w)\cup(e-v_1)$.
            \end{itemize}
            Finally if $e=1$ we let $V=V_1\cup\kappa(m-1,1)$ and $V=V_1\cup \kappa(m,e)$ otherwise.
        \end{enumerate}
        Now, finally we define $t_k=\sigma(V\cup(e+v_3),d(k))$.
    \end{algorithm}


    \begin{remark}
        In Buzzard's algorithm, we notice that there is a step when we add slopes equal to $e$ in the quantity of the number of cusps of $X_0(N)$. We want to emphasize that when we take the $\bar\rho$ component, no such things will happen as they are all associated to evil eisenstein series and they are related to split $\bar\rho$ components which do not appear in our setting where we assume $\bar\rho$ to be non-split.
    \end{remark}

\section{The variant slope algorithm for classical modular forms}
In this appendix, we show how the variant of the slope conjecture for abstract modular forms can carry over to a variant of the original slope conjecture of Buzzard for classical modular forms. 
\\We first take $\bar r\colon G_\QQ\to\GL_2(\FF)$(not fixed, this will vary) to be an absolutely irreducible representation but reducible when restricted to the decomposition group and equal to 
    $$\begin{pmatrix}
        \omega_1^{a+b+1} & *\neq0\\
        0 & \omega_1^b
    \end{pmatrix}$$ 
when restricted to the inertia group. We define some notation. For $k$ even, and at least $p+1$, $k_0$ be the remainder dividing $a+2b+2$ by $p-1$.\\
We are only dealing with the case when $\bar r_p$ is non-split, we define the sequence of $T_p$-slopes on the space $S_k(\Gamma_0(N))_{\bar r}$ to be the sequence $v_{\bar r}(k)$(note that we changed what goes in the subscript). Then, if we change the definitions in the notation above explaining Buzzard's conjecture as $d_{\bar r}(k)=\dim S_k(\Gamma_0(N))_{\bar r}$, $d_{p,\bar r}(k)=\dim S_k(\Gamma_0(Np))_{\bar r}$, We will sometimes use a separate notation for $d_{p,\epsilon,\bar r}(k)=\dim S_k(\Gamma_0(N)\cap\Gamma_1(p),\epsilon)_{\bar r}$ where $\epsilon$ is a power of the mod $p$ teichmuller lift and $\epsilon$ and $\bar r$ satisfies $\epsilon=\omega^{k-2-a-2b}$ as we mentioned that we will use an abuse of notation assuming we are interested in nonzero dimensional spaces of cusp forms. To use the definition used in \cref{mainthm}, we denote $B(p,N,\bar r)(k)=$the sequence $s_{\bar r}(k)$ outputted with given input $p,N,\bar r$. We note that there is a constarint on $k$ $\mod p-1$ to make the space of modular forms nonzero. We will ignore all other cases and assume we are with the right pairs of $\bar r $ and $k$.

    \begin{algorithm}
        First, assume that $\bar r_p$ is non-split. We start defining a sequence $t_{\bar r}(k)$. It will turn out $s_{\bar r}(2)=\kappa(d_{\bar r}(2),0)$ and $s_{\bar r}(k)=t_{\bar r}(k)$ for $k>2.$ For $4\leq k\leq p+1$, let $t_{\bar r}(k)=\kappa(d_{\bar{r}}(k),0)$ and $t_2=\kappa(d_p(2)-d(2),0)$. Set $k_{min}=p+3$. \\
        Now, assume that $k\geq k_{min}$ is even and we have $t_l$ for all even $l<k$. We now define $t_k$ depending on three parameters $x,y,z$. \\
        $x$ is defined as the unique positive integer such that \begin{equation*}
            p^x<k-1\leq p^{x+1}
        \end{equation*} 
        $y$ be the positive integer satisfying 
        \begin{equation*}
            p^xy<k-1\leq p^x(y+1).
        \end{equation*} 
        Set 
        \begin{equation*}
            z=1+\lfloor\frac{k-2-p^xy}{p^{x-1}}\rfloor.
        \end{equation*} 
        Then $1\leq z\leq p$. We define a sequence $V$ which are the first few slopes of $t_{\bar r}(k)$. The algorithm used for $V$ will depend on $y,z$ on the following three cases: $b+c\leq p-1,y<p-1<y+z,y=p-1$.\\
        \begin{enumerate}
            \item When $y+z\leq p-1$: We let 
                \begin{equation*}
                    \begin{split}
                        & k_1=k-y(p-1)p^{x-1}\\
                    & k_2=k-(y-1)(p-1)p^{x-1}-2(y+z-1)p^{x-1}.
                    \end{split}
                \end{equation*} 
                Set 
                \begin{equation*}
                    v_1=t_{\bar r}(k_1),\ v_2=t_{\bar r\otimes\omega^{1-B}}(k_2).
                \end{equation*}
                Define 
                \begin{equation*}
                    B=p^xy+p^{x-1}(z-1)+1,e=k-B,\epsilon=\chi^{1-B}.
                \end{equation*} 
                where $\chi$ is any Dirichlet character of conductor $p$ and order $p-1$. Finally set 
                \begin{equation*}
                    s=1+d_{p,\epsilon,\bar r}(1+e).
                \end{equation*}
                If $l(v_1)\geq s-1$, then let $V=\sigma(v_1,s-1)$. \\
                Otherwise let $V=v_1\cup(e-\sigma(v_2,s-1-l(v_1)))$.
            \item When $y<p-1<y+z$: We set 
                \begin{equation*}
                    \begin{split}
                        & k_1=k-((b+1)p^{x-1}(p-1))\\
                    & k_2=k-p^{x-1}(p-1).
                    \end{split}
                \end{equation*}
                We let $v_1=t_{\bar r}(k_1)$ and $v_2=t_{\bar r}(k_2) $, and
                define 
                \begin{equation*}
                    B=(y+1)p^{x-1}(p-1)+1,\ e=k-B.
                \end{equation*}
                Finally set 
                \begin{equation*}
                    s=1+d_{p,\bar r}(1+e),\ s_2=\lfloor (s-1)/2\rfloor,\ e_2=\lfloor e/2\rfloor.
                \end{equation*}
                If $l(v_1)\geq s-1$, let $V=\sigma(v_1,s-1)$.\\
                Else if $s-1\leq 2l(v_1)<2(s-1)$, let $V=v_1\cup(e-\sigma(v_1,s-1-l(v_1)))$.\\
                Else then define $w=\min(\sigma(v_2,l(v_1)+1,s_2),e_2)$. 
                \begin{itemize}
                    \item If $s$ is even, let $V=v_1\cup w\cup [e_2]\cup(e-1-w)\cup(e-v_1)$, 
                    \item if $s$ is odd, let $V=v_1\cup w\cup (e-1-w)\cup(e-v_1)$.
                \end{itemize}
            \item When $y=p-1$: We let 
                \begin{equation*}
                    \begin{split}
                        & k_1=k-p^x(p-1)\\
                    & k_2=k-p^{x-1}(p-1)
                    \end{split}
                \end{equation*}
                We set $v_1=t_{\bar r}(k_1)$ and $v_2=t_{\bar r}(k_2)$.
                Set 
                \begin{equation*}
                    B=p^x(p-1),\ e=k-B.
                \end{equation*}
                Next, set 
                \begin{equation*}
                    s=1+d_{p,\bar r}(1+e),s_2,e_2 \text{ as above.}
                \end{equation*}
                If $l(v_1)\geq s-1$, then we set $V=\sigma(v_1,s-1-l(v_1)$.\\
                Else if $s-1\leq 2l(v_1)<2(s-1)$, let $V=v_1\cup(e-\sigma(v_1,s-1-l(v_1)))$.\\
                Else define $w_0=\sigma(v_2,l(v_1)+1,s_2)$ and $w=\min(w_0+1,\kappa(l(w_0),e_2))$ \begin{itemize}
                    \item if $s$ is even $V=v_1\cup w\cup [e_2]\cup (e-1-w)\cup (e-v_1)$ and
                    \item if $s$ is odd $V=v_1\cup w\cup (e-1-w)\cup(e-v_1)$.
                \end{itemize} 
        \end{enumerate}
        Now, finally we define $k_3=2B-k$, $v_3=t_{\bar r\otimes\omega^{B-k}}(k_3)$ $t_k=\sigma(V\cup(e+v_3),d_{\bar r}(k))$.
    \end{algorithm}

We give a remark on why this algorithm is effective following 

    \begin{remark}
        Note that by \cite{interpol}, if we let $$\tilde H=\varprojlim_m H_1^{et}(Y(K^p(1+p^mM_2(\ZZ_p)))_{\bar\QQ},\cO)_{\fm_{\bar r}}^{cplx=1},$$ where $K^p$ is a neat tame level $K^p\subset\GL_2(\AA_f^p)$, then 
            \begin{equation}\label{absmodetcoh}
                \Hom_{\cO[[\GL_2(\ZZ_p)]]}(\tilde H),\Sym,^{k-2}\cO^{\oplus 2}\simeq H_{et}^1(Y(K^p\GL_2(\ZZ_p))_{\bar \QQ},\Sym^{k-2}(R^1\pi_*\underline{\cO}))_{\fm_{\bar r}}^{cplx=1}
            \end{equation}
            where the right hand side of \cref{absmodetcoh} is isomorphic to the space of classical modular forms. Hence the theory from \cref{varalg} and \cref{mainproof} with $\tilde H$ as above and $\epsilon=1\times\omega^{a+2b}$ is the relevant case to us, and from Remark 2.30 of \cite{liu2022local}, we have that twisting $\epsilon$ and $\tilde H$ simultaneously does not change the ghost series, this lets us twist the Galois representation along with $\epsilon$ to make $\epsilon$ of the form $1\times\omega^t$ for some $t$. If we do this operation for all the steps in \cref{varalg}, we get the algorithm in this appendix.
    \end{remark}
\printbibliography
\end{document}